\documentclass[11pt,a4paper]{article}
\usepackage{graphicx}
\usepackage{subcaption}
\usepackage{xcolor}
\usepackage{comment}
\usepackage[utf8]{inputenc}
\usepackage{comment}

\usepackage{hyperref}
\hypersetup{
    colorlinks=true,
    linkcolor= violet,
    citecolor = magenta,
    filecolor=magenta,      
    urlcolor=cyan,
}
\usepackage{siunitx,booktabs}
\usepackage{arydshln}

\usepackage{tablefootnote,threeparttable}
\usepackage{setspace,amsmath,amsthm,amsfonts,amssymb,mathtools,bm,algorithm,algpseudocode,
longtable,multirow,geometry,mathrsfs}
\usepackage[font=sf, labelfont={sf,bf}, margin=1cm]{caption}
\usepackage{booktabs}
\usepackage{tablefootnote}
\newcommand\scalemath[2]{\scalebox{#1}{\mbox{\ensuremath{\displaystyle #2}}}}

\numberwithin{Theorem}{section}
\numberwithin{Definition}{section}
\numberwithin{Lemma}{section}
\numberwithin{Algorithm}{section}
\numberwithin{equation}{section}

\newtheorem{theorem}{Theorem}[section]

\newtheorem{remark}{Remark}

\makeatother
	\title{General-purpose preconditioning for regularized interior point methods\\$\ $\\ 
\normalsize ERGO Technical Report 21-004}
\author{Jacek Gondzio \and Spyridon Pougkakiotis \and John W. Pearson}

\begin{document}

\maketitle

\begin{abstract}
\par In this paper we present general-purpose preconditioners for regularized augmented systems, and their corresponding normal equations, arising from optimization problems. We discuss positive definite preconditioners, suitable for CG and MINRES. We consider ``sparsifications" which avoid situations in which eigenvalues of the preconditioned matrix may become complex. Special attention is given to systems arising from the application of regularized interior point methods to linear or nonlinear convex programming problems.
\end{abstract}
\section{Introduction}
\par In this paper, we are concerned with applying Krylov-subspace methods for the efficient solution of systems of the following form:
\begin{equation} \label{General Saddle point system}
\underbrace{\begin{bmatrix}
-(Q + \rho I_n) & A^\top\\
A & \delta I_m
\end{bmatrix}}_{K} \begin{bmatrix}
\Delta x\\
\Delta y
\end{bmatrix}
= \begin{bmatrix}
\xi_1\\
\xi_2
\end{bmatrix},
\end{equation}
\noindent where $A \in \mathbb{R}^{m\times n}$ (with $m \leq n$), $Q \succeq 0 \in \mathbb{R}^{n\times n}$, $I_n$ is the identity matrix of size $n$, and $\delta,\ \rho > 0$.  Such systems arise in a plethora of applications \cite{BenGolLieActaNum}, which go far beyond optimization. However, in this paper we restrict the discussion to the case of regularized systems arising in Interior Point Methods (IPMs) for optimization
\cite{AltGondzio:OMS,ArmandOmheniJOTA,FriedOrbanMathProgComp,OrbanNumAlg,exIP-PMM:PougGond,REG:SaundersTomlin:TechReport,SIAMOptVander}. Due to the potential large dimensions of the systems, they are often solved by means 
of iterative techniques, usually from the family of Krylov-subspace methods \cite{GreenbaumSIAM}.
To guarantee efficiency of such methods the possibly ill-conditioned 
system (\ref{General Saddle point system}) often needs to be appropriately preconditioned 
and, indeed, there exists a rich literature which addresses the issue
(see the discussions in \cite{BellaviaetalCOAP,BenGolLieActaNum,BergGonZilCOAP,oldCOAPBergGondZill,SIMAX:Dollar,COAP:DollarGouldSchildersWathen,diSerafinoOrbanSIAMSCIJ}, and the references therein). 
\par Multiple preconditioning approaches have been developed in the literature, 
used to accelerate the associated iterative methods. These can be divided into positive definite 
(e.g. \cite{BenGolLieActaNum,BCOCOAP,GillMurrayPonceleSIMAX,MurphyGolubWathenSIAMSciComp,NotaySIAM,OLIVSORESLAA,SilvWathenSIAMNumAnal}) 
and indefinite ones
(e.g. \cite{SIMAX:Dollar,IpsenSIAMSciComp,KellerGouldWathenSIMAX,MurphyGolubWathenSIAMSciComp,SaadSchultzSIAMSciComp}). 
The latter are often employed within long-recurrence non-symmetric solvers (such as the Generalized Minimal RESidual method \cite{SaadSchultzSIAMSciComp}), while the former can be used within short-recurrence methods (such as the MINimal RESidual method \cite{PaigeSaundersSIAMNumAnal}). 
A comprehensive study of saddle point systems and their associated ``optimal" preconditioners can be found in \cite{BenGolLieActaNum}. Indefinite preconditioners are significantly more difficult to analyze and a simple spectral analysis is not sufficient to deduce their effectiveness (see \cite{StrakosGreenbaumPtakSIMAX}). On the other hand, positive definite preconditioners are often easier to analyze, and the eigenvalues of the preconditioned matrices allow one to theoretically compare different preconditioning approaches.
\par  In this paper, we are focused on systems arising from the application of regularized interior point methods for the solution of an arbitrary convex programming problem. In this case, $Q$ represents the Hessian of the primal barrier problem's objective function (or the Hessian of the Lagrangian in the nonlinear programming case), $A$ represents the constraint matrix (or the Jacobian of the constraints in the nonlinear programming case), while $\rho$ and $\delta$ are the primal and dual regularization parameters, respectively. We note that the IPM may contribute a term to the $(1,1)$ or the $(2,2)$ block of \eqref{General Saddle point system}, depending on the form of the constraints and non-negativity variables. Here we assume that the term is added in the $(1,1)$ block. For example, the matrix $Q$ may be written as $Q \coloneqq H + \Theta^{-1}$, where $H$ is the Hessian of the Lagrangian and $\Theta \coloneqq X Z^{-1}$ is a diagonal IPM scaling matrix (assuming $x,\ z$ are the primal and dual non-negative variables, respectively, while $X,\ Z$ denote the diagonal matrices with diagonal entries taken from vectors $x,\ z$, respectively) which originates from the use of the logarithmic barrier. 
\par We present positive definite preconditioning approaches that can be used 
within MINRES \cite{PaigeSaundersSIAMNumAnal} or the Conjugate Gradient 
method \cite{HestenesSteifelPCG}, and we provide some basic spectral 
analysis results for the associated preconditioned systems. 
More specifically, we consider preconditioners which are derived 
by ``sparsifications" of system (\ref{General Saddle point system}), 
that is, by dropping specific entries from sparse matrices $Q$ and $A$, 
thus making them more sparse and hence easier to factorize. 
Various such approaches have been proposed to date and include: preconditioners which exploit an early guess of a basic--nonbasic partition 
of variables to drop columns from $A$ \cite{OLIVSORESLAA}, 
constraint preconditioners \cite{oldCOAPBergGondZill,SIMAX:Dollar,COAP:DollarGouldSchildersWathen,diSerafinoOrbanSIAMSCIJ}, 
inexact constraint preconditioners \cite{BergGonZilCOAP} which drop specific 
entries in the matrices $Q$ and $A$, and of course a plethora of preconditioners 
which involve various levels of incomplete Cholesky factorizations 
of the matrix in \eqref{General Saddle point system}, see for example \cite{BCOCOAP}. 
The literature on preconditioners is growing rapidly and we refer the interested reader 
to \cite{BenziCompPhys,BenGolLieActaNum,DappuzoSimonediSerafinoCOAP,PearsonPestanaGAMM,WathenActaNum} and the references therein, for a detailed discussion. 
\par We consider dropping off-diagonal entries of $Q$, but restrict the elimination 
of entries in $A$ only to the removal of complete columns. Additionally, we consider sparsifying parts of rows of the Schur complement corresponding to system \eqref{General Saddle point system}. 
Such a strategy guarantees that we avoid situations in which eigenvalues 
of the preconditioned matrix may become complex (such as those employed 
in \cite{BergGonZilCOAP}), which as a consequence would have required us to employ 
non-symmetric Krylov methods. 
\par In order to construct the preconditioners, by following \cite{NLAA:BergGondMartPearPoug}, we take advantage of the properties 
of the logarithmic barrier, that allow us to know in advance which columns 
of the problem matrix are important and which are less influential. In particular, assuming the aforementioned representation of $Q$ as $Q = H + \Theta^{-1}$, the logarithmic barrier indicates which variables of the problem are likely to be inactive in the solution. More precisely, the variables are naturally split into ``basic"--$\mathcal{B}$ (not in the simplex sense), ``non-basic"--$\mathcal{N}$, and ``undecided"--$\mathcal{U}$. Hence, as IPMs progress towards optimality, we expect the following partition of the diagonal barrier matrix $\Theta^{-1}$:
\begin{align*} 
\forall j \in \mathcal{N}: \left(\Theta^{\left(\mathcal{N},\mathcal{N}\right)}\right)^{-1} = \bm{\Theta}\left(\mu^{-1}\right),\quad
\forall j \in \mathcal{B}:  \left(\Theta^{\left(\mathcal{B},\mathcal{B}\right)}\right)^{-1} = \bm{\Theta}\left(\mu\right),\quad  \forall j \in \mathcal{U}: \left(\Theta^{\left(\mathcal{U},\mathcal{U}\right)}\right)^{-1} = \bm{\Theta}\left(1\right),
\end{align*} 
\noindent where $\mu$ is the barrier parameter (and is such that $\mu \rightarrow 0$), $\mathcal{N}$, $\mathcal{B}$ and $\mathcal{U}$ are mutually disjoint, and $\mathcal{N}\cup\mathcal{B}\cup \mathcal{U} = \{1,\ldots,n\}$, while $\bm{\Theta}(\cdot)$ denotes that two positive quantities are of the same order of magnitude (see the notation section at the end of the introduction). Given the large magnitude of the diagonal elements of $Q$ for any $j \in \mathcal{N}$ (assuming $\mu$ is close to zero), we expect that the corresponding columns of $A$ (i.e. $A^{(:,\mathcal{N})}$) will not contribute important information, and thus can be set to zero when constructing a preconditioner for \eqref{General Saddle point system}. In \cite{NLAA:BergGondMartPearPoug}, the Hessian was approximated by its diagonal. In this paper, we extend this work by allowing the utilization of non-diagonal Hessian information. More specifically, we showcase how to analyze, construct, and apply the inverse of preconditioners in which we only drop non-diagonal elements of $Q$ corresponding to diagonal elements in $\mathcal{N}$. We should note at this point that such a splitting of $Q$ occurs in other second-order optimization methods as well, such as those based on augmented Lagrangian strategies (e.g. see \cite{SIAMOpt:Lietal}).
\par Furthermore, we discuss some approaches for dealing with problems for which 
the matrix $A$ may contain a subset of dense columns or rows. Any dense column or row in $A$ can pose great difficulties when trying to factorize the associated saddle-point matrix (or a preconditioner approximating it). Thus, it is desirable to alleviate the dangers of such columns or rows, by appropriate ``sparsifications" of the preconditioner, allowing us to reduce the memory requirements of applying its inverse. 
\par All such ``sparsifications" are captured in a general result presented 
in Section \ref{Section General Purpose Preconditioner} which provides 
the spectral analysis of an appropriate preconditioned normal equations' matrix. The main theorem sheds 
light on consequences of sparsifying rows of the normal equations corresponding to \eqref{General Saddle point system}, or dropping columns of $A$, and demonstrates 
that the former might produce a larger number of non-unit eigenvalues. In Section \ref{Section: regularized saddle point matrices}, these normal equation approximations are utilized to construct positive definite block-diagonal preconditioners for the saddle point system in \eqref{General Saddle point system}, and the spectral properties of the resulting preconditioned matrices are also discussed. Additionally, an alternative saddle-point preconditioner based on a $LDL^\top$ decomposition is presented.
\par All of the preconditioning approaches discussed are compared numerically on saddle-point systems arising from the application of a regularized IPM for the solution of real-life linear and convex quadratic programming problems. 
In particular, we present some numerical results on certain test problems 
taken from the Netlib (see \cite{Netlib_collection}) and the Maros--Mészáros (see \cite{OMS:MarosMeszaros}) collections. 
Then, we test the preconditioners on examples of Partial Differential 
Equation (PDE) optimal control problems. All preconditioning approaches 
have been implemented within an Interior Point-Proximal Method of Multipliers 
(IP-PMM) framework, which is a polynomially convergent primal-dual regularized IPM, 
based on the developments in \cite{exIP-PMM:PougGond,IP-PMM_SDP:PougGond}. 
A robust implementation is provided.
\par It is worth stressing that the proposed preconditioners are {\it general} 
and do not assume the knowledge of special structures which might be 
present in the matrices $Q$ and $A$ (such as block-diagonal, block-angular, 
network, PDE-induced, etc.). Therefore they may be applied 
within general-purpose IPM solvers for linear and convex quadratic programming problems.

\paragraph{Notation.} Throughout this paper we use lowercase Roman and Greek letters to  indicate vectors and scalars. Capitalized Roman fonts are used to indicate matrices. Superscripts are used to denote the components of a vector/matrix. Sets of indices are denoted by caligraphic capital fonts. For example, given $M \in \mathbb{R}^{m\times n}$, $v \in \mathbb{R}^n$, $\mathcal{R} \subseteq \{1,\ldots,m\}$, and $\mathcal{C} \subseteq \{1,\ldots,n\}$, we set
$v^{\mathcal{C}} := (v^i)_{i\in\mathcal{C}}$ and $M^{(\mathcal{R},\mathcal{C})} := ( m^{(i,j)} )_{i\in\mathcal{R}, j\in\mathcal{C}}$,
where $v^i$ is the $i$-th entry of $v$ and $m^{(i,j)}$ the $(i,j)$-th entry of $M$. Additionally, the full set of indices is denoted by a colon. In particular, $M^{(:,\mathcal{C})}$ denotes all columns of $M$ with indices in $\mathcal{C}$.  Given a matrix $M$, we denote the diagonal matrix with the same diagonal elements as $M$ by $\textnormal{Diag}(M)$.
We use $\lambda_{\min}(B)$ ($\lambda_{\max}(B)$, respectively) to denote the minimum (maximum) eigenvalue
of an arbitrary square matrix $B$ with real eigenvalues. Similarly, $\sigma_{\min}(B)$ ($\sigma_{\max}(B)$, respectively) denotes
the minimum (maximum) singular value of an arbitrary rectangular matrix $B$. 
We use $0_{m,n}$ to denote a matrix of size $m\times n$ with entries equal to $0$.
Furthermore, we use $I_n$ to indicate the identity matrix of size $n$. For any finite set $\mathcal{A}$, we denote by $\vert \mathcal{A} \vert$
its cardinality. Finally, given two positive functions $T,\ f \colon \mathbb{R}_{+} \mapsto \mathbb{R}_{+}$, we write $T(x) = \bm{\Theta}(f(x))$ if these functions are of the same order of magnitude, that is, there exist constants $c_1,\ c_2 > 0$ and some $x_0 \geq 0$ such that $c_1 f(x) \leq T(x) \leq c_2 f(x)$, for all $x \geq x_0$.
\paragraph{Structure of the article.} The rest of this paper is organized as follows. In Section \ref{Section General Purpose Preconditioner} we present some preconditioners suitable for the normal equations. Then, in Section \ref{Section: regularized saddle point matrices}, we adapt these preconditioners to regularized saddle point systems. Subsequently, in Section \ref{Section Regularized IPM linear systems} we focus on saddle point systems arising from the application of regularized IPMs to convex programming problems, and present some numerical results. Finally, in Section \ref{Section Conclusions}, we deliver our conclusions.
\section{Regularized normal equations} \label{Section General Purpose Preconditioner}
\par We begin by defining the regularized normal equations matrix (or Schur complement) $M \coloneqq A G A^\top + \delta I_m \in \mathbb{R}^{m\times m}$, corresponding to \eqref{General Saddle point system}, where $G \coloneqq (Q+\rho I_n)^{-1}  \succ 0 \in \mathbb{R}^{n\times n}$. In this section, we derive and analyze preconditioning approaches for $M$. As we have already mentioned in the introduction, we achieve a simplification of the preconditioner by setting to zero certain columns of the matrix $A$ (and consequently sparsifying the corresponding rows and columns of $Q$), as well as by sparsifying certain rows of the matrix $M$.
\par More specifically, we define two integers, $k_c$ and $k_r$, such that $0 \leq k_c \leq n$ and $0\leq k_r \leq m$. The former counts the number of columns of $A$ (and corresponding columns and rows of $Q$) that are set to zero (that are sparsified, respectively), while the latter counts the number of rows of the matrix $M$ that are sparsified. At this point, we assume that we have been given these columns or rows, but we later specify how these can be chosen (see Remark \ref{Remark: how to choose columns or rows} and Section \ref{Section Regularized IPM linear systems}). In order to highlight these given columns and rows, we assume that we are given two permutation matrices; a column permutation $\mathscr{P}_c \in \mathbb{R}^{n\times n}$, and a row permutation $\mathscr{P}_r \in \mathbb{R}^{m\times m}$. Applied to the constraint matrix $A$ in \eqref{General Saddle point system}, these permutations 
bring all the columns and rows which will need to be treated 
specially to the leading positions of columns and rows 
of $ \mathscr{P}_r A \mathscr{P}_c$, respectively.
\par Given the previous permutation matrices, let us firstly define an approximation to the matrix $Q$. In particular, we approximate $Q$ by the following block-diagonal and positive semi-definite matrix:
\begin{equation} \label{eqn: approximation of Q}
\widehat{Q} \coloneqq \mathscr{P}_c\begin{bmatrix}
\widehat{Q}_1 & 0_{k_c,(n-k_c)}\\
0_{n-k_c,k_c} & \widehat{Q}_2
\end{bmatrix} \mathscr{P}_c^\top, \quad \text{assuming}\quad Q \equiv \mathscr{P}_c \begin{bmatrix}
Q_1 & Q_3^\top\\
Q_3 & Q_2
\end{bmatrix}\mathscr{P}_c^\top,
\end{equation}
\noindent where $\widehat{Q}_1 = Q_1$ or $\widehat{Q}_1 = \textnormal{Diag}(Q_1)$, and $\widehat{Q}_2 = Q_2$ or $\widehat{Q}_2 = \textnormal{Diag}(Q_2)$ (both cases are treated concurrently). The column permutation $\mathscr{P}_c$ reorders symmetrically both rows and columns of the matrix $Q$ in \eqref{General Saddle point system}, and places the $k_c$ columns and rows which will be sparsified at the leading $(1,1)$ block of the permuted version of the matrix $Q$. Given this approximation of $Q$, let us define an approximated normal equations' matrix that will be of interest when analyzing the spectral properties of the preconditioned matrices derived in this paper:
\begin{equation} \label{eqn: definition of approximated normal equations}
\widehat{M} \coloneqq A \widehat{G}A^\top + \delta I_m,\quad \widehat{G} \coloneqq \left(\widehat{Q}+\rho I_{n}\right)^{-1} \equiv \mathscr{P}_c\begin{bmatrix}
\left(\widehat{Q}_1+ \rho I_{k_c}\right)^{-1} & 0_{k_c,(n-k_c)}\\
0_{(n-k_c),k_c} & \left( \widehat{Q}_2 + \rho I_{n-k_c}\right)^{-1}
\end{bmatrix} \mathscr{P}_c^\top.
\end{equation}
\par In what follows, we derive a preconditioner for the approximated normal equations' matrix $\widehat{M}$. We should note that system \eqref{General Saddle point system} is solved using the normal equations only if $Q$ is diagonal (due to obvious numerical considerations), in which case $Q \equiv \widehat{Q}$, and thus $M \equiv \widehat{M}$. If this is not the case, we would like to derive a preconditioner for the approximate normal equations' matrix $\widehat{M}$. Later on, and in particular in Section \ref{Section: regularized saddle point matrices}, this is utilized to derive and analyze a preconditioner for the matrix $K$, defined in \eqref{General Saddle point system}.
\par We proceed by introducing some notation, for convenience of exposition. Given the definitions in \eqref{eqn: approximation of Q}, \eqref{eqn: definition of approximated normal equations}, we can write: 
\[B \coloneqq \mathscr{P}_r A \widehat{G}^{\frac{1}{2}}\mathscr{P}_c = \begin{bmatrix}
B_{11}& B_{12} \\
B_{21} & B_{22}
\end{bmatrix}, \]
\noindent where $\mathscr{P}_r$ is a given row-permutation matrix, $B_{11} \in \mathbb{R}^{k_r \times k_c}$, $B_{12} \in \mathbb{R}^{k_r \times (n-k_c)}$, $B_{21} \in \mathbb{R}^{(m-k_r)\times k_c}$, and $B_{22} \in \mathbb{R}^{(m-k_r)\times(n-k_c)}$. Notice that the aim of the row-permutation matrix $\mathscr{P}_r$, is to bring on top all rows of the matrix $\widehat{M}$, that we are planning to sparsify in order to construct the preconditioner. Let us further introduce the following notation:
\[\mathscr{P}_r \widehat{M} \mathscr{P}_r^\top \equiv \begin{bmatrix}
\widehat{M}_{11} & \widehat{M}_{21}^\top\\
\widehat{M}_{21} & \widehat{M}_{22}
\end{bmatrix}, \]
\noindent where $\widehat{M}_{11} $, $\widehat{M}_{21} $, and $\widehat{M}_{22} $ are defined as:
\begin{align*}
&\widehat{M}_{11} \coloneqq B_{11}B_{11}^\top + B_{12}B_{12}^\top + \delta I_{k_r}  \ \in \mathbb{R}^{k_r\times k_r},\\
&\widehat{M}_{21} \coloneqq B_{21}B_{11}^\top + B_{22}B_{12}^\top \ \in \mathbb{R}^{(m-k_r)\times k_r},\\
&\widehat{M}_{22} \coloneqq B_{21}B_{21}^\top + B_{22}B_{22}^\top + \delta I_{m-k_r} \ \in \mathbb{R}^{(m-k_r)\times(m-k_r)}.
\end{align*}
\par In what follows, we present two preconditioning strategies for the matrix $\widehat{M}$. Both approaches exploit the sparsification of the matrix $\widehat{M}$. The first approach relies on a Cholesky decomposition of a sparsified matrix, while the second approach is based on an $LDL^\top$ decomposition of a sparsified augmented system matrix, which is used to implicitly derive a preconditioner for $\widehat{M}$.
\subsection{A Cholesky-based preconditioner} \label{Subsection: Preconditioner for RNE}
\par Our first proposal is to consider preconditioning $\mathscr{P}_r \widehat{M} \mathscr{P}_r^\top$ with the following matrix:
\begin{equation} \label{Preconditioner for regularized normal equations}
P_{NE,(k_c,k_r)} \coloneqq \begin{bmatrix}
\widehat{M}_{11} & 0_{k_r, (m-k_r)} \\
0_{(m-k_r),k_r} & \widetilde{M}_{22}
\end{bmatrix}, \qquad \widetilde{M}_{22} \coloneqq \widehat{M}_{22} - B_{21}B_{21}^\top.
\end{equation}
\noindent The notation $P_{NE,(k_c,k_r)}$ signifies that this is a preconditioner for the normal equations, in which we drop $k_c$ columns from the matrix $A$ and sparsify $k_r$  rows of the matrix $\widehat{M}$. Notice that if $k_c = 0$ (that is, we only sparsify certain rows of $\widehat{M}$ to construct the preconditioner), we can write $\small{B \equiv \mathscr{P}_r A \widehat{G}^{\frac{1}{2}} = \begin{bmatrix}
B_{12} \\
B_{22}
\end{bmatrix}}$, while $B_{11}$, $B_{21}$ are zero-dimensional, and hence absent. In this case, we have
\[P_{NE,(0,k_r)} \coloneqq \begin{bmatrix}
\widehat{M}_{11} & 0_{k_r, (m-k_r)} \\
0_{(m-k_r),k_r} & \widehat{M}_{22}
\end{bmatrix}. \]
\noindent On the other hand, if $k_r = 0$ (that is, we only drop $k_c$ columns from $A$ to construct the preconditioner), we have $B \equiv \begin{bmatrix}
B_{21} & B_{22}
\end{bmatrix}$, and $B_{11},\ B_{12}$ are absent. Then, we obtain
\[P_{NE,(k_c,0)} = \widetilde{M}_{22}. \]
\noindent Notice that the latter is obtained since $\widehat{Q}$ is block-separable (with respect to the permutation $\mathscr{P}_c$), and thus dropping the respective $k_c$ columns of $A$, results in dropping $B_{21}B_{21}^\top$. For simplicity of notation, for the rest of this subsection we set $P_{NE,(k_c,k_r)} \equiv P_{NE}$.
\par In the following theorem, we analyze the spectrum of the preconditioned matrix $P_{NE}^{-1}\mathscr{P}_r \widehat{M} \mathscr{P}_r^\top$, with respect to the spectrum of the associated matrices. 
\begin{theorem} \label{Spectral Analysis theorem}
The preconditioned matrix $P_{NE}^{-1}\mathscr{P}_r \widehat{M} \mathscr{P}_r^\top$ has at least 
$\max\{m-(2k_r+k_c),0\}$ eigenvalues at $1$. If $k_c > 0$ and $k_r > 0$, all remaining eigenvalues lie in the following interval:
\[ I_{k_c,k_r} \coloneqq  \left[\frac{\delta}{\delta+\sigma_{\max}^2(B)}, 2 + \frac{\lambda_{\max}(B_{21}B_{21}^\top)}{\delta+\lambda_{\min}(B_{22}B_{22}^\top)}\right].\]
\noindent If $k_c > 0$ and $k_r = 0$, the previous interval reduces to 
\[I_{k_c} \coloneqq \left[1, 1 + \frac{\lambda_{\max}(B_{21}B_{21}^\top)}{\delta+\lambda_{\min}(B_{22}B_{22}^\top)}\right],\]
\noindent while if $k_r > 0$ and $k_c = 0$, we obtain
\[ I_{k_r} \coloneqq  \left[\frac{\delta}{\delta+\sigma_{\max}^2(B)}, 2 \right].\]
\end{theorem}
\begin{proof}
\par Given an arbitrary eigenvalue $\lambda$ (which must be positive since $P_{NE} \succ 0$ and $\widehat{M} \succ 0$) corresponding to a unit eigenvector $v$, let us write the generalized eigenproblem as:
\begin{equation} \label{Spectral Analysis: Generalized Eigenproblem}
\begin{bmatrix}
\widehat{M}_{11} &\widehat{M}_{21}^\top\\
\widehat{M}_{21} & \widehat{M}_{22}
\end{bmatrix} \begin{bmatrix}
v_1\\
v_2
\end{bmatrix} = \lambda \begin{bmatrix}
\widehat{M}_{11}v_1\\
\widetilde{M}_{22} v_2
\end{bmatrix}.
\end{equation}
\noindent We separate the analysis into two cases.
\par \textbf{Case 1:} Let $v_2 \in \textnormal{Null}(\widehat{M}_{21}^\top)$. Firstly, we notice that:
\[ \textnormal{dim}\left(\textnormal{Null}\left(\widehat{M}_{21}^\top\right)\right) = (m-k_r) - \textnormal{rank}\left(\widehat{M}_{21}^\top\right) \geq \max\{m-2k_r,0\}. \]
\noindent Two sub-cases arise here. For the first sub-case, we notice that if $v_1 \neq 0$, then from positive definiteness of $\widehat{M}_{11}$, combined with the first block equation of \eqref{Spectral Analysis: Generalized Eigenproblem}, we obtain that $\lambda = 1$. In turn, we claim that this implies that $v_2 \in \textnormal{Null}(B_{21}B_{21}^\top)$ and $v_1 \in \textnormal{Null}(\widehat{M}_{21})$. To see this, assume that $v_2 \notin \textnormal{Null}(B_{21}B_{21}^\top)$. Then from the second block equation of \eqref{Spectral Analysis: Generalized Eigenproblem} we obtain:
\[\widehat{M}_{21}v_1 + \widehat{M}_{22} v_2 = \widetilde{M}_{22} v_2 ~~ \Rightarrow ~~ \widehat{M}_{21}v_1 = -B_{21}B_{21}^\top v_2, \]
\noindent where we used the definition of $\widetilde{M}_{22}$. If $v_2 \notin \textnormal{Null}(B_{21}B_{21}^\top)$, this implies that $v_2^\top B_{21}B_{21}^\top v_2  > 0$. The previous equation then yields that
\[v_2^\top \widehat{M}_{21} v_1 = - v_2^\top B_{21}B_{21}^\top v_2 ~~ \Rightarrow ~~ 0 = -v_2^\top B_{21}B_{21}^\top v_2 < 0,\]
\noindent which follows from the base assumption (i.e. $v_2 \in \textnormal{Null}(\widehat{M}_{21}^\top)$), and results in a contradiction. Hence, $v_2 \in \textnormal{Null}(B_{21}B_{21}^\top)$. On the other hand, if $v_1 \notin \textnormal{Null}(\widehat{M}_{21})$ then the second block equation yields directly a contradiction, since we have shown that $v_2 \in \textnormal{Null}(B_{21}B_{21}^\top)$.
\par Next we consider the second sub-case, i.e. $v_1 = 0$. Combined with our base assumption, the first block equation of \eqref{Spectral Analysis: Generalized Eigenproblem} becomes redundant. From the second block equation of the eigenproblem, and using $v_1 = 0$, we obtain:
\begin{equation*}
\begin{split}
v_2^\top  \widehat{M}_{22} v_2 = &\ \lambda v_2^\top\widetilde{M}_{22} v_2 \\
\Rightarrow~~v_2^\top\left( \widetilde{M}_{22} +  B_{21}B_{21}^\top\right)v_2 = &\ \lambda v_2^\top \widetilde{M}_{22} v_2.
\end{split}
\end{equation*}
\noindent Hence we have that:
\[ \lambda  = 1 + \frac{v_2^\top (B_{21}B_{21}^\top)v_2}{v_2^\top \widetilde{M}_{22} v_2} \leq 1 + \frac{\lambda_{\max}(B_{21}B_{21}^\top)}{\delta + \lambda_{\min}(B_{22}B_{22}^\top)}. \]
\par All eigenvalues in this case can be bounded by the previous inequality and there will be at most $\textnormal{rank}(B_{21}B_{21}^\top)$ non-unit eigenvalues. On the other hand, if $v_2 \in \textnormal{Null}(B_{21}B_{21}^\top)$, then trivially $\lambda = 1$. This concludes the first case. Notice that this case would occur necessarily if $k_r = 0$, and thus, we obtain the interval $I_{k_c}$.
\par \textbf{Case 2:} In this case, we assume that $v_2 \notin \textnormal{Null}(\widehat{M}_{21}^\top)$. In what follows we assume $\lambda \neq 1$ (noting that $\lambda = 1$ would only occur if $v_1 \in \textnormal{Null}(\widehat{M}_{21})$ and $v_2 \in \textnormal{Null}(B_{21}B_{21}^\top)$), and there are at most $2k_r$ such eigenvalues. Given the previous assumption, and using the first block equation in \eqref{Spectral Analysis: Generalized Eigenproblem}, we obtain:
\[v_1 = \frac{1}{\lambda-1}\widehat{M}_{11}^{-1}\widehat{M}_{21}^\top v_2.\]
\noindent Substituting the previous into the second block equation of \eqref{Spectral Analysis: Generalized Eigenproblem}, yields the following generalized eigenproblem:
\begin{equation} \label{Spectral Analysis: Case 2 eigenproblem}
\begin{split}
\left(\widehat{M}_{21}\widehat{M}_{11}^{-1}\widehat{M}_{21}^\top  + (\lambda - 1)B_{21}B_{21}^\top \right) v_2 = (\lambda-1)^2 \widetilde{M}_{22}v_2,
\end{split}
\end{equation}
\noindent where we used the definitions of $\widetilde{M}_{22}$ and $\widehat{M}_{22}$. Premultiplying \eqref{Spectral Analysis: Case 2 eigenproblem} by $v_2^\top$ and rearranging yields the following quadratic algebraic equation that $\lambda$ must satisfy in this case:
\begin{equation} \label{Spectral Analysis: Case 2 eigenvalue quadratic relation}
\lambda^2 + \beta \lambda + \gamma = 0,
\end{equation}
\noindent where 
\[\beta \coloneqq -2 - \frac{v_2^\top B_{21}B_{21}^\top v_2}{v_2^\top \widetilde{M}_{22} v_2}\] 
\noindent and 
\[ \gamma \coloneqq 1 - \frac{v_2^\top \left(\widehat{M}_{21}\widehat{M}_{11}^{-1}\widehat{M}_{21}^\top - B_{21}B_{21}^\top \right)v_2}{v_2^\top \widetilde{M}_{22}v_2}.\]
\par Let us notice that the smallest eigenvalue is at least as large as $\frac{\delta}{\delta + \sigma_{\max}^2(B)}$. This follows from positive definiteness of $P_{NE}$ and $M$, and the bound can be deduced by noticing that $\lambda_{\min}(P_{NE}^{-1}\widehat{M}) \geq \frac{\lambda_{\min}(\widehat{M})}{\lambda_{\max}(P_{NE})} \geq \frac{\delta}{\delta + \sigma_{\max}^2(B)}$. Still we need to find an upper bound for the largest eigenvalue. To that end, notice that:
\[ \gamma = \frac{v_2^\top \left( \widehat{M}_{22} - \widehat{M}_{21}\widehat{M}_{11}^{-1}\widehat{M}_{21}^\top \right) v_2}{v_2^\top \widetilde{M}_{22} v_2},\]
\noindent which follows from the definition of $\widetilde{M}_{22}$. Positive definiteness of $\widehat{M}$ then implies that $\gamma > 0$. From the last relation we also have that:
\[0 < \gamma \leq 1 + \frac{v_2^\top B_{21}B_{21}^\top v_2}{v_2^\top \widetilde{M}_{22}v_2} = \frac{v_2^\top \widehat{M}_{22} v_2}{v_2^\top \widetilde{M}_{22} v_2} \leq 1+ \frac{\lambda_{\max}(B_{21}B_{21}^\top)}{\delta + \lambda_{\min}(B_{22}B_{22}^\top)} \eqqcolon \gamma_u. \]
\noindent Furthermore, $\beta_l \coloneqq -\left(2 + \frac{\lambda_{\max}\left(B_{21}B_{21}^\top\right)}{\delta + \lambda_{\min}\left(B_{22}B_{22}^\top\right)}\right) \leq \beta \leq -2.$ From the previous inequality, one can also observe that $\gamma < - \beta - 1$.
\par Returning to \eqref{Spectral Analysis: Case 2 eigenvalue quadratic relation}, we first consider the following solution:
\[ \lambda_- = \frac{1}{2}\left(-\beta -\sqrt{\beta^2-4\gamma}\right).\]
\noindent It is easy to see that $\beta^2 -4\gamma$ is always larger than $0$. Next, we notice that the relation for $\lambda_-$ is increasing with respect to $\gamma$. We omit finding a lower bound for $\lambda_-$ since this was established earlier. For the upper bound, we use the fact that $\gamma < -\beta - 1$, to obtain:
\[ \lambda_- < \frac{1}{2}\left(-\beta -\sqrt{\beta^2 + 4(\beta+1)}\right) = \frac{1}{2}\left(|\beta| - |\beta+2|\right) =  1,\]
\noindent since $\beta \leq -2$ (also, in the beginning of this case, we have treated $\lambda_- = 1$ separately). 
\par Finally, we consider the other solution of \eqref{Spectral Analysis: Case 2 eigenvalue quadratic relation}, which reads:
\[\lambda_{+} = \frac{1}{2}\left(-\beta + \sqrt{\beta^2 - 4\gamma}\right). \]
\noindent  Firstly, we can easily notice that $\lambda_{+} > 1$. Subsequently, upon noticing that $\lambda_{+}$ is decreasing with respect to $\gamma$, we can obtain the following obvious bound:
\[ \lambda_+ \leq |\beta| \leq -\beta_l.\]
\noindent Now, let us observe that dropping $\widehat{M}_{21}$ and $\widehat{M}_{21}^\top$ yields at most $k_r + \textnormal{rank}(\widehat{M}_{21}^\top) \leq 2k_r$ eigenvalue outliers. Similarly, dropping $B_{21}B_{21}^\top$ from the $(2,2)$ block of $M$ yields at most $\textnormal{rank}\left(B_{21}\right) \leq k_c$ eigenvalue outliers. Hence, there will be at least $\max\left\{m-(2k_r+k_c),0\right\}$ eigenvalues of the preconditioned matrix at $1$.
\par Finally, the case where $k_c = 0$ and $k_r > 0$ follows by a direct generalization of \cite[Theorem 4.1]{DeSimDiSerGondPougViol}, and completes the proof. \qed
\end{proof}

\begin{remark} \label{Remark: how to choose columns or rows}
Now that we have presented the spectral properties of the preconditioned system, let us discuss the use of such a preconditioning strategy. In practice, we solve system \eqref{General Saddle point system} using the normal equations, only when matrix $Q$ is diagonal. As already mentioned, in this case $\widehat{M} = M$, and thus $P_{NE}$ is a preconditioner for the normal equations' matrix. In Section \ref{Section: regularized saddle point matrices}, we discuss how $P_{NE}$ is utilized to construct preconditioners for the saddle point matrix in \eqref{General Saddle point system}, even if $Q$ is not diagonal (in which case $\widehat{M}$ is an approximation of the normal equations' matrix). 
\par Let us now discuss how to choose which columns of $A$ (or rows of $\widehat{M}$, respectively) to drop (sparsify, respectively).
\begin{itemize}
\item[$\bullet$] Firstly, it often happens in optimization, and especially when solving systems arising from the application of an interior point method (as already mentioned in the introduction), to have certain diagonal elements of $G$ that are very small. In view of this property, and given the bound presented in Theorem \textnormal{\ref{Spectral Analysis theorem}}, we can observe that dropping all columns corresponding to small diagonal elements in $G$ results in manageable and not too sizeable outliers. Such a preconditioner was proposed in \textnormal{\cite{NLAA:BergGondMartPearPoug}} for the case where $\widehat{Q} = \textnormal{Diag}(Q)$, and arises as a special case of $P_{NE}$ in \textnormal{\eqref{Preconditioner for regularized normal equations}}, by choosing $k_r = 0$ and a suitable permutation matrix $\mathscr{P}_c$, traversing first the $k_c$ indices corresponding to the smallest diagonal elements of $G$.
\item[$\bullet$] Secondly, it is common in many application areas to have a small number of columns or rows of $A$ that are dense. Such columns (or rows) could pose significant difficulties as they produce dense factors when one tries to factorize the normal equations (e.g. using a Cholesky decomposition). This is especially the case for dense columns.  A single dense column of $A$ with $p$ non-zero entries induces a dense window of size $p\times p$ in the normal equations (we refer the reader to the discussion in \textnormal{\cite[Section 4]{mybib:AGMX}}). The use of a preconditioner like the one defined in \textnormal{\eqref{Preconditioner for regularized normal equations}} serves the purpose of dropping (sparsifying, respectively) such columns of $A$ (rows of $\widehat{M}$, respectively), thus making the Cholesky factors of $P_{NE}$ significantly more sparse. For example, we may find two permutation matrices $\mathscr{P}_{r}$, $\mathscr{P}_{c}$ which sort the rows and columns, respectively, of $A$ in descending order of their number of non-zeros, and write $\widehat{A} = \mathscr{P}_{r} A \mathscr{P}_{c}$. Then, the resulting normal equations read as $\mathcal{P}_{r}^\top B B^\top \mathcal{P}_{r} + \delta I_m$. As long as the number of dropped columns or rows is low (which is observed in several applications), the number of outliers produced by this dropping strategy is manageable. While some of these outliers will be dangerously close to zero (given that the regularization parameter $\delta > 0$ is small), they can be dealt with efficiently. We should note, however, that if $\widehat{Q}$ is non-diagonal, without a sparse representation of its inverse, this strategy would be unattractive to employ, and thus in this work, we consider it only when $\widehat{Q}$ is diagonal. Indeed, as we discuss in Section \textnormal{\ref{Subsection: ldl-based preconditioner}}, in this case we never explicitly form the normal equations. Instead, we appropriately utilize a $LDL^\top$ decomposition, and the fill-in produced by few dense rows or columns of $A$ can be prevented.
\end{itemize}
\end{remark}
\begin{remark} 
As we discuss later, a case of interest would be to only drop certain  $k_c$ columns of $A$, which results in introducing at most $k_c$ eigenvalue outliers in the preconditioned matrix. Similarly, only sparsifying certain ($k_r$) rows of $\widehat{M}$, introduces at most $2k_r$ non-unit eigenvalues.  Notice that dropping columns is expected to be more useful in general, resulting in fewer outliers and possibly in greater gains (either in terms of processing time or memory requirements). Indeed, notice that, on the one hand, dropping dense columns of $A$ can result in a significant reduction of the fill-in within a factorization of the normal equations, while, on the other hand, dropping any column of $A$ corresponding to a small diagonal element of $G$ yields a not too sizeable outlier. However, in certain special applications one has to resort to sparsifying ``problematic'' rows. Indeed, we refer the reader to \textnormal{\cite[Section 4]{DeSimDiSerGondPougViol}}, where such a row-sparsifying strategy was key to the efficient solution of fMRI classification problems.
\end{remark}
\subsection{An $LDL^\top$-based preconditioner} \label{Subsection: ldl-based preconditioner}
\par Next, we present an alternative to the preconditioner in \eqref{Preconditioner for regularized normal equations}. This approach offers an alternative for dealing with a small set of dense columns or rows of the matrix $A$, while remaining efficient when the approximate Hessian $\widehat{Q}$, given in \eqref{eqn: approximation of Q}, is non-diagonal. More specifically, let us divide the columns of the matrix $A$ into two mutually exclusive sets $\mathcal{B}$ and $\mathcal{N}$, based solely on the magnitude of the respective diagonal elements of the matrix $G$, and not on the density of the columns of $A$. Then, using the column-dropping strategy presented in the previous section (assuming that the variables corresponding to $\mathcal{N}$ are such that $Q^{(j,j)} \geq Q^{(i,i)}$, for all $j \in \mathcal{N}$ and all $i \in \mathcal{B}$), we propose approximating the matrix $\widehat{M}$ by the following preconditioner:
\begin{equation} \label{Normal equations preconditioner without dropping dense row/col}
 {P}_{NE,\left(\lvert \mathcal{N}\rvert,0 \right)} = A^{(:,\mathcal{B})} \widehat{G}^{(\mathcal{B},\mathcal{B})}\left(A^{(:,\mathcal{B})}\right)^\top + \delta I_m,
 \end{equation}
\noindent which was proposed in \cite{NLAA:BergGondMartPearPoug}, for the special case where $\widehat{G}$ was diagonal. Notice that the block-separable structure of $\widehat{Q}$, given in \eqref{eqn: approximation of Q}, implies that $(\widehat{G}^{(\mathcal{B},\mathcal{B})})^{-1} = \widehat{Q}^{(\mathcal{B},\mathcal{B})} + \rho I_{|\mathcal{B}|}$. Given our previous discussion, we would like to avoid applying the inverse of this preconditioner by means of a Cholesky decomposition, as a single dense column of $A$ in $\mathcal{B}$ could result in dense Cholesky factors, while the potential non-diagonal nature of $\widehat{Q}^{\left(\mathcal{B},\mathcal{B}\right)}$ might prevent us from even efficiently forming it. Instead, we form an appropriate saddle point system to compute the action of the approximated normal equations. More specifically, given an arbitrary vector $y \in \mathbb{R}^m$, instead of computing ${P}_{NE,\left(\lvert \mathcal{N}\rvert,0 \right)}^{-1}y$ using a Cholesky decomposition, we can compute 
\begin{equation} \label{NE Preconditioner using LDL}
\underbrace{\begin{bmatrix}
-\left(Q^{(\mathcal{B},\mathcal{B})} + \rho I_{|\mathcal{B}|}\right) & \left(A^{(:,\mathcal{B})}\right)^\top\\
A^{(:,\mathcal{B})} & \delta I_m 
\end{bmatrix}}_{\widetilde{P}_{NE}} \begin{bmatrix}
w_1\\
w_2
\end{bmatrix} = \begin{bmatrix}
0_{|\mathcal{B}|}\\
y
\end{bmatrix},
\end{equation}
\noindent by means of an $LDL^\top$ decomposition of the previous saddle point matrix. Then, we notice that returning $w_2$ is equivalent to computing ${P}_{NE,\left(\lvert \mathcal{N}\rvert,0 \right)}^{-1}y$.
\par Following the discussion in \cite[Section 4]{mybib:AGMX}, we know that using an $LDL^\top$ structure to factorize the matrix in \eqref{NE Preconditioner using LDL} can result in significant memory savings compared to the Cholesky decomposition of ${P}_{NE,\left(\lvert \mathcal{N}\rvert,0 \right)}$. Notice that in view of the regularized nature of the systems under consideration (indeed, we have assumed that $G$ is positive definite), we can use the result in \cite{SIAMOptVander}, stating that matrices like the one in \eqref{NE Preconditioner using LDL} are \emph{quasi-definite}; any symmetric permutation of such matrices admits an $LDL^\top$ decomposition. 
\par While this approach might seem expensive, it can provide significant time and memory savings, especially in cases where $A^{(:,\mathcal{B})}$ contains dense columns. In the previous section we discussed a strategy for alleviating this issue, noting however that such a strategy can only be used to deal with a small number of dense columns. On the contrary, if we have a sizeable subset of the columns of $A^{(:,\mathcal{B})}$ that are dense, we could delay their pivot order within the $LDL^\top$, thus significantly reducing the overall fill-in of the decomposition factors, without introducing any eigenvalue outliers in the preconditioned system. Of course, finding the optimal permutation for the $LDL^\top$ decomposition is an NP-hard problem, however, there have been developed several permutation heuristics tailored to such symmetric decompositions. Moreover, in the $LDL^\top$ factorization, the pivots are computed dynamically to ensure both stability and sparsity. In view of the previous, the preconditioner based on solving \eqref{NE Preconditioner using LDL} is expected to be more stable than its counterpart based on the Cholesky decomposition. Finally, difficulties arising from dense rows or in general ``problematic" rows can also be alleviated using a heuristic proposed in \cite{EJOR:MarosMeszaros}.
\par On the other hand, by using this approach we avoid explicitly forming the preconditioner $P_{NE,\left(\lvert \mathcal{N}\rvert,0\right)}$. This is especially important in cases where $\widehat{Q}$ is non-diagonal, and thus forming $P_{NE}$ can be extremely expensive. Hence, the approach presented in this subsection allows us to utilize non-diagonal information in a practical way when constructing an approximation for the matrix $Q$.
\section{Regularized saddle point matrices} \label{Section: regularized saddle point matrices}
\par Let us now consider the regularized saddle point system in \eqref{General Saddle point system}. In what follows, we discuss two families of preconditioning strategies, noting their advantages and disadvantages. All presented preconditioners are positive definite in order to be usable within the MINRES method, which is a short-recurrence iterative solver, suitable for solving symmetric indefinite or quasi-definite systems. This allows us to avoid non-symmetric long-recurrence solvers like the GMRES method.
\subsection{Block-diagonal preconditioners} \label{Subsection: block diagonal preconditioners}
\par The most common approach is to employ a block-diagonal preconditioner (see \cite{BenGolLieActaNum,NLAA:BergGondMartPearPoug,NotaySIAM,SilvWathenSIAMNumAnal}).
To construct such a preconditioner, we need approximations for the $(1,1)$ block $F \coloneqq Q + \rho I_n$ of the coefficient matrix in \eqref{General Saddle point system}, and for its associated Schur complement $M = A \left(Q +\rho I_n\right)^{-1}A^\top + \delta I_m$. 
\par In this section, we assume that $Q$ is approximated as shown in \eqref{eqn: approximation of Q}, and thus can potentially contain non-diagonal blocks. Concerning the approximation of the Schur complement matrix $M$, we can employ the preconditioner $P_{NE,\left(k_c,k_r\right)}$ given in \eqref{Preconditioner for regularized normal equations}. Then, we can define the following preconditioner for the coefficient matrix $K$ in \eqref{General Saddle point system}:
\begin{equation} \label{Block diagonal saddle point preconditioner}
P_{AS,(k_c,k_r)} \coloneqq \begin{bmatrix}
\widehat{Q} + \rho I_n & 0_{n,m}\\
0_{m,n} & P_{NE,\left(k_c,k_r\right)}
\end{bmatrix} \equiv \begin{bmatrix}
\widehat{F}  & 0_{n,m}\\
0_{m,n} & P_{NE,\left(k_c,k_r\right)}
\end{bmatrix}.
\end{equation}
\noindent For the rest of this subsection, let $P_{AS,(k_c,k_r)} \equiv P_{AS}$ and $P_{NE,\left(k_c,k_r\right)} \equiv P_{NE}$.
\par In order to analyze the spectrum of the preconditioned matrix $P_{AS}^{-1} K$, we introduce some notation for simplicity of exposition. We work with positive definite similarity transformations of the associated matrices, defined as
			\begin{equation}
				\label{simsym}
				\widetilde{F} =  \widehat{F}^{-1/2} F  \widehat{F}^{-1/2}, \quad 
				\widetilde{M}_{NE} = P_{NE}^{-1/2} \widehat{M} P_{NE}^{-1/2}, 
\end{equation}
\noindent where $\widehat{M} \coloneqq A \widehat{F}^{-1} A^\top + \delta I_m$. Then, we set
		\[
		\begin{array}{lcllcllcl}
			\alpha_{NE} &=& \lambda_{\min} \left( \widetilde{M}_{NE}\right),~   & ~\beta_{NE} &=& \lambda_{\max} \left( \widetilde{M}_{NE}\right),~
& ~\kappa_{NE} &=& \dfrac{\beta_{NE}}{\alpha_{NE}}, 
			\\[.6em]
\alpha_{F} &=& \lambda_{\min} \left(\widetilde{F}\right),~ &   ~\beta_{F} &=& \lambda_{\max} \left(\widetilde{F}\right),~
& ~\kappa_{F} &=& \dfrac{\beta_{F}}{\alpha_{F}}. 
\end{array} \]
From the definition of $\widehat{Q}$ given in \eqref{eqn: approximation of Q}, we can observe that $\alpha_{F} \le 1 \le \beta_{F}$ as
		\[ \frac 1n \sum_{i=1}^n \lambda_i \left(\widehat{F}^{-1} F\right) = \frac 1n \textnormal{Tr}\left(\widehat{F}^{-1} F\right)
 = 1.\]
On the other hand, notice that $\alpha_{NE}$ and $\beta_{NE}$ are bounded directly from Theorem \ref{Spectral Analysis theorem}. Indeed, from \eqref{simsym}, we observe that we need to bound the spectrum of an approximate preconditioned Schur complement matrix, since $M$ has been substituted by $\widehat{M}$. This is exactly what the analysis in Section \ref{Subsection: Preconditioner for RNE} does. Below we provide a theorem analyzing the spectral properties of the matrix $P_{AS}^{-1}K$. 
		\begin{theorem} \label{Theorem Spectral analysis augmented system}
		The eigenvalues of $P_{AS}^{-1} K$ lie in the union of the following intervals:
			\[ I_- = \left[- \beta_{F} -\sqrt{\beta_{NE}}, -\alpha_{F}\right]; \quad I_+ = \left[\frac{1}{2}\left( -\beta_{F} + \sqrt{\beta_{F}^2 + 4\alpha_{NE}}\right) , 1 + \sqrt{\beta_{NE}-1}\right].
			\]
	\end{theorem}
\begin{proof}
The proof, which follows by trivially extending \cite[Theorem 3]{NLAA:BergGondMartPearPoug}, is summarized in Appendix \ref{Appendix: proof of spectral analysis for augmented system} for completeness.
\end{proof}
\par The authors of \cite{NLAA:BergGondMartPearPoug} make use of the approximation $\widehat{Q} = \textnormal{Diag}(Q)$, and ${P}_{NE} = {P}_{NE,\left(\lvert \mathcal{N}\rvert,0\right)}$ (where the latter is defined in \eqref{Normal equations preconditioner without dropping dense row/col}, with $\widehat{G} = (\widehat{Q} + \rho I_n)^{-1}$). Approximating $Q$ by its diagonal allows one to form the normal equations' preconditioner (i.e. ${P}_{NE,\left(\lvert \mathcal{N}\rvert,0\right)}$), thus enabling the efficient use of a Cholesky factorization. However, there might exist problems for which a better approximation of the matrix $Q$ is required. In this case, one could consider a sparsified version of $Q$ like the one given in \eqref{eqn: approximation of Q}. While this could lead to reasonable approximations, the problem of fill-in introduced by $(\widehat{Q} + \rho I_n)^{-1}$ in the Schur complement approximation, ${P}_{NE,\left(\lvert \mathcal{N}\rvert,0\right)}$, would (in general) remain.
\par In order to address the previous issue, we make use of the $LDL^\top$-based preconditioner defined in Section \ref{Subsection: ldl-based preconditioner}. As in Section \ref{Subsection: ldl-based preconditioner}, we divide the columns of $A$ using the sets $\mathcal{B}$ and $\mathcal{N}$, where $\mathcal{N}$ contains all indices corresponding to the largest diagonal elements of $Q$. Assume that $Q$ is approximated by $\widehat{Q}$ as given in \eqref{eqn: approximation of Q}, where the permutation matrix $\mathscr{P}_c$ traverses first the column indices in $\mathcal{N}$. For example, a reasonable approximation would be the following
\begin{equation} \label{approximation for matrix Q in non-separable case}
\mathscr{P}_c^\top \widehat{Q} \mathscr{P}_c = \begin{bmatrix}
\textnormal{Diag}\left(Q^{(\mathcal{N},\mathcal{N})}\right) & 0_{\lvert \mathcal{N}\rvert, \lvert \mathcal{B}\rvert}\\
0_{\lvert \mathcal{B}\rvert, \lvert \mathcal{N}\rvert} & Q^{(\mathcal{B},\mathcal{B})}
\end{bmatrix}. 
\end{equation}
\noindent The effect of this approximation in the context of regularized IPMs has been analyzed in detail in \cite{JOTA:PougGond}. We notice that $(Q^{(\mathcal{B},\mathcal{B})} + \rho I_{|\mathcal{B}|})^{-1}$ does not introduce significant fill-in in the $(2,2)$ block of the preconditioner in \eqref{Block diagonal saddle point preconditioner}, as we implicitly invert this block using the methodology presented in Section \ref{Subsection: ldl-based preconditioner}. 
\begin{remark}
Notice that further approximations can be employed here. In particular, we could define a banded approximation of $Q$ and then employ the approximation proposed earlier. The implicit inversion of the Schur complement, outlined in Section \textnormal{\ref{Subsection: ldl-based preconditioner}}, gives us complete freedom on how to approximate $Q$, and hence we no longer rely on diagonal approximations. We return to this point in the numerical experiments.
\end{remark}

\subsection{Factorization-based preconditioners} \label{Subsection: factorization-based preconditioners for saddle point systems}

\par Finally, given the regularized nature of the systems under consideration, we can construct fac\-to\-ri\-za\-tion-based preconditioners for MINRES. In particular, we can compute $K = L D L^\top$ (with $K$ in \eqref{General Saddle point system}), where $D$ is a diagonal matrix (since $K$ is quasi-definite \cite{SIAMOptVander}) with $n$ negative and $m$ positive elements on its diagonal. Then, by defining $P_K \coloneqq L|D|^{\frac{1}{2}}$, the preconditioned saddle point matrix reads:
\[P_K^{-1} K P_K^{-\top} = |D|^{-1}D,\]
\noindent and hence contains only two distinct eigenvalues: $-1$ and $1$ \cite{GillMurrayPonceleSIMAX,OrbanNumAlg}. As before, let us assume that we have available a splitting of the columns of $A$ such that $ A\mathscr{P}_{c} = [A^{\mathcal{B}}\ \ A^{\mathcal{N}}]$, where $\mathcal{B}$ contains indices corresponding to the smallest diagonal elements of $Q$. Then, we can precondition $K$, left and right, by $\widehat{P}_K \coloneqq \widehat{L} |\widehat{D}|^{\frac{1}{2}}$, where $\widehat{K} = \widehat{L} \widehat{D}\widehat{L}^\top$ and:
\begin{equation*}
\widehat{K} \coloneqq \begin{bmatrix}
-\widehat{Q} & \widehat{A}^\top\\
\widehat{A} & \delta I_m,
\end{bmatrix},
\end{equation*}
\noindent with $\widehat{A} \coloneqq [A^{\mathcal{B}}\ \ 0_{m,|\mathcal{N}|}]\mathscr{P}_{c}^\top$, and $\widehat{Q}$ defined as in \eqref{approximation for matrix Q in non-separable case}. Notice that by setting several columns of $A$ to zero, as well as by sparsifying the respective rows and columns of $Q$, the cost of applying the inverse of $\widehat{K}$ is significantly reduced when compared to that required to apply the inverse of $K$.
\par Further limited-memory versions of this preconditioner can be employed, e.g. by using the methodologies presented in \cite{OrbanNumAlg,ScottTumaNumAlg}. Other approximations of the blocks of $\widehat{K}$, based on the structure of the problem at hand, could also be possible, as already mentioned in the previous subsection.
\par We should note, however, that this approach is less stable than the approach presented in Section \ref{Subsection: block diagonal preconditioners}. This is because we are required to use only diagonal pivots during the $LDL^\top$ decomposition for this methodology to work (indeed, notice that the presence of a non-diagonal matrix $D$ in the factorization of $K$ would not allow the use of such a preconditioning strategy). If the regularization parameters $\delta$ or $\rho$ have very small values, the stability of the factorization could be compromised (since we enforce the use of only diagonal $1\times 1$ pivots), and we would have to heavily rely on stability introduced by means of uniform \cite{REG:SaundersTomlin:TechReport} or weighted regularization \cite{AltGondzio:OMS}. On the other hand, the methodology presented in Section \ref{Subsection: block diagonal preconditioners} would not be affected by the occasional use of $2\times 2$ pivots within the $LDL^\top$ factorization for the implicit inversion of the approximate Schur complement. Of course the latter is not the case if the ``Analyze" phase of the factorization (used to determine the pivot order) is performed separately, however, the subset of columns in $\mathcal{B}$ may change significantly from one iteration to the next, making this strategy less attractive. Nevertheless, this factorization-based approach can be more efficient than the approach presented in Section \ref{Subsection: block diagonal preconditioners}, when solving certain non-separable convex programming problems. This is because the approach in Section \ref{Subsection: block diagonal preconditioners} requires the computation of an $LDL^\top$ decomposition of the coefficient matrix $\widetilde{P}_{NE}$ in \eqref{NE Preconditioner using LDL} (with potential $2\times2$ pivots) as well as a Cholesky decomposition of $\widehat{Q} + \rho I_n$ (or some iterative scheme which could be application-dependent, as in \cite{PearsonPorcelliStollNLAA}).

\section{Regularized IPMs: numerical results} \label{Section Regularized IPM linear systems}
\par Let us now focus on the case of the regularized saddle point systems (and their respective normal equations) arising from the application of regularized IPMs on convex programming problems. The MATLAB code, which is based on the IP-PMM presented in \cite{exIP-PMM:PougGond,IP-PMM_SDP:PougGond}, can be found on GitHub\footnote{\url{https://github.com/spougkakiotis/IP-PMM_QP_Solver}}.
\par In all the presented experiments, a 6-digit accurate solution is requested. The reader is referred to \cite[Section 4]{NLAA:BergGondMartPearPoug}, and \cite[Section 5]{exIP-PMM:PougGond} for the implementation details of the algorithm (such as termination criteria, the employed predictor--corrector scheme for the solution of the Newton system, as well as the tuning of the algorithmic regularization parameters). The associated iterative methods (i.e. PCG or MINRES) are adaptively terminated if the following accuracy is reached: $\frac{\min\{10^{-3},\max\{10^{-1}\cdot\mu_k,\texttt{tol}\}\}}{\max\{1,\|\texttt{rhs}\|\}}$, where $\texttt{tol} = 10^{-6}$, $\mu_k$ is the barrier parameter at iteration $k$, and $\texttt{rhs}$ is the right hand side of the system being solved. This adaptive stopping criterion is based on the developments in \cite{Cafieri_et_al_COAP}. When PCG is employed we allow at most 100 iterations per linear system solved, while for MINRES up to 200 iterations are allowed. If the maximum number of Krylov iterations is reached, an inexact Newton direction is accepted if it is at least 3-digit accurate. Any Cholesky decomposition is computed via the \texttt{chol} function of MATLAB. When an $LDL^\top$ decomposition is employed, we utilize the \texttt{ldl} function of MATLAB. In this case, the minimum pivot threshold is adaptively set to $\text{pivot}_{\text{thr}} = 0.1\cdot\min\{\delta,\rho,10^{-4}\}$. This is done to ensure that no $2\times 2$ pivots are used during the factorization, ensuring that the factorization remains efficient. However, in the context of the preconditioner in Section \ref{Subsection: ldl-based preconditioner}, where $2\times 2$ pivots can safely be used (unlike the preconditioner presented in Section \ref{Subsection: factorization-based preconditioners for saddle point systems}, which requires the use of $1\times 1$ pivots), this mechanism is turned off when $\min\{\delta,\rho\} \leq 10^{-8}$, and we set $\text{pivot}_{\text{thr}} = 10^{-6}$ to ensure stability. All the presented experiments were run on a PC with a 2.2GHz Intel Core i7 processor (hexa-core), 16GB RAM, run under the Windows 10 operating system. The MATLAB version used was 2019a. 
\subsection{Linear programming} \label{Subsection: LP IPM numerics}
\par Let us initially focus on  Linear Programming (LP) problems of the following form:
\begin{equation} \label{non-regularized primal LP} \tag{LP}
\underset{x \in \mathbb{R}^n}{\text{min}} \  c^\top x , \ \ \text{s.t.}  \  Ax = b,   \ x^{\mathcal{I}} \geq 0,\ x^{\mathcal{F}}\ \textnormal{free}, 
\end{equation}
\noindent where $A \in \mathbb{R}^{m\times n}$, $\mathcal{I} \cap \mathcal{F} = \emptyset$, and $\mathcal{I}\cup \mathcal{F} = \{1,\ldots,n\}$. Applying regularized IPMs to problems like \eqref{non-regularized primal LP}, one often solves a regularized normal equations system at every iteration. Such systems have a coefficient matrix of the following form: 
\begin{equation*}
M = A G A^\top + \delta I_m,\qquad G^{(i,i)} = \begin{cases}
												\frac{1}{\rho} &\ \textnormal{if}\ i \in\mathcal{F}, \\
												\frac{1}{\rho + z^i/x^i} &\ \textnormal{if}\ i \in \mathcal{I},
												\end{cases}
\end{equation*}
\noindent where $\delta,\ \rho > 0$ and $z \in \mathbb{R}^n$ (where $z^{\mathcal{I}} \geq 0$, $z^{\mathcal{F}} = 0$) are the dual slack variables. Notice that the IPM penalty parameter $\mu$ is often tuned as $\mu = \frac{(x^{\mathcal{I}})^\top z^{\mathcal{I}}}{n}$ and we expect that $\mu \rightarrow 0$. As already mentioned in the introduction, the variables are naturally split into ``basic"--$\mathcal{B}$, ``non-basic"--$\mathcal{N}$, and ``undecided"--$\mathcal{U}$. Hence, as IPMs progress towards optimality, we expect the following partition of the quotient $\frac{x^{\mathcal{I}}}{z^{\mathcal{I}}}$: 
\begin{align*} 
&\forall j \in \mathcal{N}: x^j \rightarrow 0,\quad z^j \rightarrow \widehat{z}^j> 0 & \Rightarrow \quad &  \frac{x^j}{z^j } = \frac{x^j z^j}{(z^j)^2} = \bm{\Theta}(\mu),\\
&\forall j \in \mathcal{B}: x^j \rightarrow \widehat{x}^j > 0,\quad z^j \rightarrow 0  & \Rightarrow \quad &  \frac{x^j}{z^j } = \frac{(x^j)^2 }{x^j z^j} = \bm{\Theta}(\mu^{-1}),\\
&\forall j \in \mathcal{U}: x^j = \bm{\Theta}(1), \quad z^j = \bm{\Theta}(1) & \Rightarrow \quad & \frac{x^j}{z^j} = \bm{\Theta}(1),
\end{align*} 
\noindent where $\mathcal{N}$, $\mathcal{B}$, and $\mathcal{U}$ are mutually disjoint, and $\mathcal{N}\cup\mathcal{B}\cup \mathcal{U} = \mathcal{I}$. For the rest of this section, we assume that $\delta = \bm{\Theta}(\rho) = \bm{\Theta}(\mu)$. This assumption is based on the developments in \cite{exIP-PMM:PougGond,IP-PMM_SDP:PougGond}, where a polynomially convergent regularized IPM is derived for convex quadratic and  linear positive semi-definite programming problems, respectively. Following \cite{NLAA:BergGondMartPearPoug}, we could precondition the matrix $M$ using the following matrix:
\begin{equation} \label{LP NE dropping preconditioner}
P_{NE,\left(\lvert \mathcal{N}\rvert,0\right)} = A^{\left(\colon,\mathcal{R}\right)}G^{\left(\mathcal{R},\mathcal{R}\right)}\left(A^{\left(\colon,\mathcal{R}\right)}\right)^\top + \delta I_m,
\end{equation}
\noindent where $\mathcal{R} \coloneqq \mathcal{F}\cup\mathcal{B}\cup\mathcal{U}$. Then, by \cite[Theorem 1]{NLAA:BergGondMartPearPoug} (or by applying Theorem \ref{Spectral Analysis theorem}), we obtain:  
\[\lambda_{\max}\left(P_{NE,\left(\lvert \mathcal{N}\rvert,0\right)}^{-1}M\right) \leq 1 + \frac{\underset{j \in \mathcal{N}}{\max} \left(G^{(j,j)}\right)}{\delta}\sigma_{\max}^2(A),\qquad \lambda_{\min}\left(P_{NE,\left(\lvert \mathcal{N}\rvert,0\right)}^{-1}M\right) \geq 1. \]
\noindent The preconditioner in \eqref{LP NE dropping preconditioner} is a special case of the preconditioner defined in Section \ref{Section General Purpose Preconditioner}. Indeed, as already indicated by its notation, it can be derived by setting $k_r = 0$ and then by dropping all columns belonging to $\mathcal{N}$, i.e. we set $k_c = |\mathcal{N}|$ and we drop the $k_c$ columns of $A$ corresponding to the smallest diagonal elements of $G$. Notice that in the linear programming case, the matrix $\widehat{M}$ that is analyzed in Section \ref{Section General Purpose Preconditioner} coincides with $M$, since $G$ is diagonal.
\par From our previous remarks, we notice that 
\[\underset{j \in \mathcal{N}}{\max} \left(G^{(j,j)}\right)=\bm{\Theta}(\mu) = \bm{\Theta}(\delta)\]
\noindent implies that the spectrum of the preconditioned matrix remains bounded and is asymptotically independent of $\mu$ (assuming that $\delta = \bm{\Theta}(\mu)$). While this preconditioner performs very well in practice (see \cite[Section 4]{NLAA:BergGondMartPearPoug}), it can be expensive to compute in certain cases, as its inverse needs to be applied by means of a Cholesky decomposition. To that end, we propose to further approximate this matrix as indicated in Section \ref{Section General Purpose Preconditioner}. This idea is based on the fact that the Preconditioned Conjugate Gradient (PCG) method is expected to converge in a small number of iterations, if the preconditioned system matrix can be written as:
\[ P^{-1}M = I_m + U + V, \]
\noindent where $P$ is the preconditioner, $M$ is the normal equations matrix, $U$ is a low-rank matrix, and $V$ is a matrix with small norm. In our case, dropping the part of the normal equations corresponding to $\mathcal{N}$ contributes the small-norm term (that is, $V = A^{\left(\colon,\mathcal{N}\right)}G^{\left(\mathcal{N},\mathcal{N}\right)}(A^{\left(\colon,\mathcal{N}\right)})^\top$, the norm of which is of the order of magnitude of $\mu$), and furthermore dropping a few dense columns (or sparsifying certain rows) contributes the low-rank term (indeed, as already shown in Theorem \ref{Spectral Analysis theorem}, dropping $k_c$ dense columns of $A$ and sparsifying $k_r$ dense rows of $M$ yields at most $2k_r + k_c$ outliers, and thus $\textnormal{rank}(U) \leq 2k_c + k_r$, where, in this case, $k_c$ does not account for columns corresponding to the index set $\mathcal{N}$).
\par To construct such a preconditioner, we first need to note that $\mathcal{R}$ will change at every IPM iteration. However, we can heuristically choose which columns to drop (and/or rows to sparsify) based on the sparsity pattern of $A$. To that end, at the beginning of the optimization procedure, we count the number of non-zeros of each column and row of $A$, respectively. These can then be used to sort the columns and rows of $A$ in descending order of their number of non-zero entries. These sorted columns and rows can easily be represented by means of two permutation matrices $\mathscr{P}_c$ and $\mathscr{P}_r$. We note that this is a heuristic, and it is not guaranteed to identify the most ``problematic'' columns or rows (which can be sources of difficulty for IPMs). For a discussion on such heuristics, and alternatives, the reader is referred to \cite[Section 4]{mybib:AGMX}, and the references therein.
\subsubsection{Numerical results}
\par Initially, we present some results to show the effect of dropping dense columns of $A$ and then of sparsifying dense rows of $M$ using the strategy outlined in Section \ref{Section General Purpose Preconditioner}. Then, we present a comparison between the preconditioner in \eqref{Preconditioner for regularized normal equations} (that is, $P_{NE,(k_c,k_r)}$), the preconditioner given in \eqref{Normal equations preconditioner without dropping dense row/col} or \eqref{LP NE dropping preconditioner} (denoted as $P_{NE,\left(\lvert \mathcal{N}\rvert ,0\right)}$), noting that \eqref{Normal equations preconditioner without dropping dense row/col} and \eqref{LP NE dropping preconditioner} are equivalent, and the one in \eqref{NE Preconditioner using LDL} (that is, $\widetilde{P}_{NE}$). Notice that in the linear programming case, employing $\widetilde{P}_{NE}$ and $P_{NE,\left(\lvert \mathcal{N}\rvert ,0\right)}$ should yield identical results in exact arithmetic. The difference between these preconditioning strategies, is that in the latter case (i.e. $\widetilde{P}_{NE}$) the action of the former preconditioner (i.e. $P_{NE,\left(\lvert \mathcal{N}\rvert ,0\right)}$) is computed implicitly by means of an $LDL^\top$ factorization of $\widetilde{P}_{NE}$, as indicated in Section \ref{Subsection: ldl-based preconditioner}.
\paragraph{Dropping dense columns versus factorizing directly.} We run IP-PMM on all problems from the Netlib collection that have some dense columns, where dense is defined in this case to be a column with at least 15\% non-zero elements. We note that these were the only problems within the Netlib collection having any columns with such a density of non-zero elements. We compare an IP-PMM using Cholesky factorization for the solution of the associated Newton system (Chol.), with an IP-PMM that uses the preconditioner $P_{NE,\left(k_c,0\right)}$ presented in Section \ref{Section General Purpose Preconditioner} alongside PCG. The latter method is only allowed to drop dense columns (at most 30) to create the preconditioner (and thus $k_c \leq 30$). The results are collected in Table \ref{Table: dropping columns}, where $k_c$ denotes the number of dense columns that are dropped to create the preconditioner, \textbf{nnz} denotes the number of non-zero elements present in the Cholesky factor, \textbf{Avg. Krylov It.} denotes the average number of Krylov iterations performed from the inexact approach, while \textbf{Krylov Last} denotes the number of Krylov iterates performed in the last IP iteration of the inexact approach.
\begin{table}[!ht]
\centering
\caption{The effect of dropping dense ($>15\%$) columns of $A$ (Netlib collection).\label{Table: dropping columns}}
\scalebox{0.9}{
\begin{tabular}{llllrrrrrr}
     \toprule
    \multirow{2}{*}{\textbf{Name}}  & \multirow{2}{*}{$\mathbf{k_c}$}   & \multicolumn{2}{c}{\textbf{nnz}}      & \multicolumn{2}{c}{\textbf{time (s)}}  & \multirow{2}{*}{\textbf{Avg. Krylov It.}} & \multirow{2}{*}{\textbf{Krylov Last}} \\   \cmidrule(l{2pt}r{2pt}){3-4} \cmidrule(l{2pt}r{2pt}){5-6}
  & & {Chol.} & {$P_{NE,\left(k_c,0\right)}$} & {Chol.} & {$P_{NE,\left(k_c,0\right)}$} & & \\ \midrule
 BLEND  & 5 & 1,006 & 736 & 0.05 & 0.05 & 14.50 & 14 \\
 FIT1P & $20$ & 197,676 & 26,706 &1.16 & 0.32 & 18.67 &  35 \\
 FIT2P  & 16 & 4,516,500 & 1,962,616 & 27.17 & 16.21 & 17.33 & 33 \\
 FORPLAN &  $8$ & 3,810 & 2,918 & 0.11 & 0.11 & 2.81 & 3  \\
 ISRAEL  & $30$ & 12,261 & 1,744 & 0.14 & 0.30 & 53.50 & 100 \\
 SEBA  & $14$ & 55,937 & 2,238 & 0.35 & 0.18 & 13.83 &  15\\
  \bottomrule
\end{tabular}}
\end{table}
\par From Table \ref{Table: dropping columns} we can immediately see that certain dense columns present in the constraint matrix $A$ can have a significant impact on the sparsity pattern of the Cholesky factors. This is a well-known fact (see for example the discussion in \cite[Section 4]{mybib:AGMX}). Notice that the Netlib collection contains only small- to medium-scale instances. For such problems, memory is not an issue, and hence direct methods tend to be faster than their iterative alternatives (like PCG). Despite the small size of the presented problems, we can see tremendous memory savings  (and even a decrease in CPU time) for problems FIT1P, FIT2P, and SEBA, by eliminating only a small number of dense columns. On the other hand, for problems where we observe an increase in CPU time (e.g. see ISRAEL), the memory savings can be significant, making this acceptable.
\paragraph{Sparsifying dense rows versus factorizing directly.} Next, we consider the case where the inexact version of IP-PMM is only allowed to sparsify dense rows, where dense is defined in this case to be a row with at least 25\% non-zero elements. 
\par Before moving to the numerical results, let us note some differences between sparsifying rows of $M$ and dropping columns of $A$. Firstly, as we have shown in Section \ref{Section General Purpose Preconditioner}, sparsifying $k$ rows can potentially introduce twice as many outliers, while dropping $k$ columns introduces at most $k$ eigenvalue outliers. Furthermore, the potential density induced in the Cholesky factors by a single dense column is usually more significant than that introduced by a single dense row. However, we cannot know in advance how effective the dropping of a column will be. On the other hand, sparsifying dense rows of $M$ introduces a certain separability in the approximate normal equations matrix, allowing us to estimate very well the memory savings.
\par In Table \ref{Table: dropping rows} we compare the direct IP-PMM, to its inexact version, the preconditioner of which is only allowed to sparsify at most 30 dense rows of $M$. Any problem with at least one dense row from the Netlib collection is considered.
\begin{table}[!ht]
\centering
\caption{The effect of dropping dense ($>25\%$) rows of $M$ (Netlib collection).\label{Table: dropping rows}}
\scalebox{0.9}{
\begin{tabular}{llllrrrr}
     \toprule
    \multirow{2}{*}{\textbf{Name}}      & \multirow{2}{*}{$\mathbf{k_r}$}   & \multicolumn{2}{c}{\textbf{nnz}}      & \multicolumn{2}{c}{\textbf{time (s)}} & \multirow{2}{*}{\textbf{Avg. Krylov It.}} & \multirow{2}{*}{\textbf{Krylov Last}} \\   \cmidrule(l{2pt}r{2pt}){3-4} \cmidrule(l{2pt}r{2pt}){5-6}
  & & {Chol.} & {$P_{NE,\left(0,k_r\right)}$} & {Chol.} & {$P_{NE,\left(0,k_r\right)}$}\\ \midrule
 BEACONFD & $17$ & 2,903 & 1,475 &0.07 & 0.09 & 13.59 & 23  \\
 BLEND  & 1 & 1,006 & 959 & 0.05 & 0.04 & 2.88&  3\\
 D6CUBE  & 6 & 55,179 & 52,757 & 0.27 & 0.44 & 12.73 & 18 \\
 FIT1D &  $11$ & 14,726 & 4,973 & 0.22 & 0.38 & 23.50 & 46  \\
 FIT2D  & $12$ & 139,843 & 49,513 & 1.51 & 2.47&  21.30& 44 \\
 ISRAEL  & $3$ & $12,261$ & $11,758$ & 0.14 & 0.14& 6.96 & 8 \\
 STANDATA & 1 & 3,416 & 3,168 & 0.07 & 0.07& 2.94 & 3 \\
 STANDGUB & 1 & 3,418& 3,170 & 0.08 & 0.08& 2.98 & 3\\
 STANDMPS  & 1 & 5,529 & 5,185 & 0.09 & 0.09&2.91 & 3 \\	
 WOOD1P  & $27$ & 19,088 & 13,879 & 0.34 & 0.44 & 14.95& 13 \\
  \bottomrule
\end{tabular}}
\end{table}
\par From Table \ref{Table: dropping rows} we can observe that the required memory to form the Cholesky factors is consistently decreased but CPU time is often increased by the row-dropping strategy. We should note that an increase in CPU time usually relates to the size of the problems under consideration, and CPU time as well as a memory advantages can be observed if the problem is sufficiently large with sufficiently many dense rows. In particular, this row sparsifying strategy was successfully used within IP-PMM in \cite[Section 4]{DeSimDiSerGondPougViol} in order to tackle fMRI sparse approximation problems (in which the constraint matrix contains thousands of dense rows). Memory requirements were significantly lowered, allowing this inexact version to outperform its exact counterpart, while being competitive with standard state-of-the-art first-order methods used to solve such problems. 
\paragraph{The Cholesky versus the $LDL^\top$ approach.} Let us now provide some numerical evidence for the potential benefits and drawbacks of the approach presented in Section \ref{Subsection: ldl-based preconditioner} over that presented in Section \ref{Subsection: Preconditioner for RNE}. To that end, we run three inexact versions of IP-PMM, on some of the most challenging instances within the Netlib collection. The first approach uses the preconditioner given in \eqref{Preconditioner for regularized normal equations} (denoted as $P_{NE,(k_c,k_r)}$, allowing at most 15 dense columns/rows to be dropped/sparsified), the second uses the preconditioner given in \eqref{Normal equations preconditioner without dropping dense row/col} (denoted as $P_{NE,\left( \lvert \mathcal{N} \rvert, 0\right)}$; notice that this is the same as the former preconditioner, without employing the strategy of dropping/sparsifying dense columns/rows), while the third version uses the preconditioner in \eqref{NE Preconditioner using LDL}, denoted as $\widetilde{P}_{NE}$. In all three cases the set $\mathcal{B}$, used to decide which columns are dropped irrespectively of their density, is determined as indicated at the beginning of this section.
\begin{table}[!ht]
\centering
\caption{Cholesky-based versus the $LDL^\top$-based preconditioner (Netlib collection).\label{Table: chol versus ldl Netlib}}
\scalebox{0.8}{
\begin{tabular}{lrrrrrrrrr}
     \toprule
    \multirow{2}{*}{\textbf{Name}}  & \multicolumn{3}{c}{\textbf{Krylov Its.}}     & \multicolumn{3}{c}{\textbf{max nnz}}      & \multicolumn{3}{c}{\textbf{time (s)}}  \\   \cmidrule(l{2pt}r{2pt}){2-4}\cmidrule(l{2pt}r{2pt}){5-7} \cmidrule(l{2pt}r{2pt}){8-10}
 & {$P_{NE,(k_c,k_r)}$} & {$P_{NE,\left( \lvert \mathcal{N} \rvert, 0\right)}$} & {$\widetilde{P}_{NE}$} & {$P_{NE,(k_c,k_r)}$} & {$P_{NE,\left( \lvert \mathcal{N} \rvert, 0\right)}$} & {$\widetilde{P}_{NE}$} & {$P_{NE,(k_c,k_r)}$} & {$P_{NE,\left( \lvert \mathcal{N} \rvert, 0\right)}$} & {$\widetilde{P}_{NE}$}\\ \midrule
 AGG & 2,966 & 2,966 & 2,268 & $1.60 \cdot 10^4$ & $1.60 \cdot 10^4$ & $1.40\cdot 10^4$ & 0.42 & 0.42 & 0.58\\
 DFL001 & 4,778 & 4,778 & 4,969 & $1.55 \cdot 10^6$ & $1.55 \cdot 10^6$ & $1.54 \cdot 10^6$ & 8.40 & 8.40 & 31.51\\
 FIT1P & 1,166 & 636 & 491 & $1.20\cdot10^5$ & $2.00\cdot10^5$ & $1.36\cdot10^4$ & 0.58 & 1.56 & 0.51\\
 FIT2P & 2,555 & 1,880 & 1,867 & $4.19\cdot10^6$ & $4.60\cdot10^6$ & $9.4\cdot10^4$ & 29.45 & 43.10 & 15.95\\
 PILOT & 3,558 & 3,558 & 3,530 & $1.99 \cdot 10^5$ & $1.99 \cdot 10^5$ & $4.02 \cdot 10^5$ & 3.79 & 3.79 & 9.16 \\
 QAP12 & 1,583 & 1,583 & 1,591 & $2.48 \cdot 10^6$ & $2.48 \cdot 10^6$ & $1.61 \cdot 10^6$ & 2.09 & 2.09 & 5.24\\
 QAP15 & 1,704 & 1,704 & 1,708 & $8.83\cdot 10^6$ & $8.83\cdot 10^6$ & $5.28 \cdot 10^6$ &20.56 & 20.56 & 25.42\\
 SEBA & 2,678 & 2,396 & 2,313 & $2.18 \cdot 10^3$ & $5.54\cdot 10^4$ & $7.5\cdot 10^3$ & 0.33 & 0.80 & 0.84\\
  \bottomrule
\end{tabular}}
\end{table}
\par From Table \ref{Table: chol versus ldl Netlib} we can observe that the $LDL^\top$-based preconditioner can provide substantial (memory and/or CPU time) benefits for certain problems (e.g. see problems FIT1P, FIT2P, QAP12, QAP15, SEBA). Nevertheless, we should note that this approach is usually slower, albeit more stable (as a pivot re-ordering is computed at every iteration, and the pivots of the $LDL^\top$ factorization are chosen to ensure stability as well as efficiency). We observe that instances AGG, DFL001, PILOT did not benefit from the use of this strategy, neither in terms of efficiency nor memory requirements, despite a comparable number of Krylov iterations. This comes in line with our observations in Section \ref{Subsection: ldl-based preconditioner}, since neither of the aforementioned instances contains any dense rows or columns. Notice that the stability and efficiency of the preconditioner $\widetilde{P}_{NE}$, depends heavily on the choice for the threshold for the $\texttt{ldl}$ function of MATLAB. Larger values imply better stability, however at the cost of efficiency, since more $2\times 2$ pivots will be chosen during the $LDL^\top$ factorization. The stability of this approach can be guaranteed by using a large-enough pivot threshold.  Additionally, there are instances without dense rows or columns (see QAP12, QAP15), in which the $LDL^\top$-based preconditioner (i.e. $\widetilde{P}_{NE}$) provides significant advantages in terms of memory requirements. Finally, we note that for problems AGG, PILOT, QAP12, and QAP15, the two Cholesky-based variants are exactly the same, as no dense columns or rows were present. 
\par There is a long-standing discussion on the comparison between the Cholesky and the $LDL^\top$ decompositions. The former tend to be faster and usually easier to implement, while the latter tend to be slower, more stable, and more general. For more on this subject, the reader is referred to \cite[Section 4]{mybib:AGMX} and the references therein.

\subsection{Convex quadratic programming}
\par Next, we consider problems of the following form:
\begin{equation} \label{QP problem} \tag{QP}
\min_{x\in \mathbb{R}^n} c^\top x + \frac{1}{2} x^\top H x, \quad \textnormal{s.t.}\ Ax = b,\ x^{\mathcal{I}} \geq 0,\ x^{\mathcal{F}} \textnormal{ free},
\end{equation}
\noindent where $H \in \mathbb{R}^{n \times n}$ is the positive semi-definite Hessian matrix. Let us notice that a similar partitioning of the variables as that presented in Section \ref{Subsection: LP IPM numerics} also holds in this case. Hence, the index set $\mathcal{N}$ guides us on which columns of $A$ to drop. In the case where $H$ is either diagonal, or can be well-approximated by a diagonal, the discussion of Section \ref{Subsection: LP IPM numerics}, about dropping dense columns of $A$ (or sparsifying dense rows of the approximated Schur complement $\widehat{M}$ given in \eqref{eqn: definition of approximated normal equations}), also applies here.
\par In what follows we make use of three different preconditioners. We compare the two block-diagonal preconditioners given in Section \ref{Subsection: block diagonal preconditioners}. The first is called $P^{C}_{AS,\left(k_c,k_r\right)}$ (where the superscript $C$ stands for Cholesky, which is used to invert the $(2,2)$ block of this preconditioner), and employs a diagonal approximation for $Q$, allowing one to drop dense columns and/or sparsify dense rows as shown in Section \ref{Subsection: Preconditioner for RNE}, and the second is called $P^{L}_{AS,\left(\lvert \mathcal{N}\rvert,0\right)}$ (where the superscript $L$ stands for $LDL^\top$), and employs a block-diagonal approximation of $Q$, using the implicit inversion of the Schur complement proposed in Section \ref{Subsection: ldl-based preconditioner}. The block-diagonal preconditioners are also compared against the factorization-based preconditioner presented in Section \ref{Subsection: factorization-based preconditioners for saddle point systems}, termed as $\widehat{P}_{K}$. 

\subsubsection{Numerical results}
\noindent In the following experiments, we employ MINRES to solve the associated Newton systems. Initially, we present the comparison of the three preconditioning strategies over some problems from the Maros--Mészáros collection of convex quadratic programming problems. Then, the two block-diagonal preconditioning approaches are compared over some Partial Differential Equation (PDE) optimization problems. 

\paragraph{Maros--Mészáros collection.}
In Table \ref{Comparison of QP preconditioners (Maros-Meszaros collection)}, we report on the runs of the three methods on a diverse set of non-separable instances within the Maros--Mészáros test set.

\begin{table}[!ht]
\centering
\caption{Comparison of QP preconditioners (Maros--Mészáros collection).\label{Comparison of QP preconditioners (Maros-Meszaros collection)}}
\scalebox{0.78}{
\begin{threeparttable}
\begin{tabular}{lrrrrrrrrr}
     \toprule
    \multirow{2}{*}{\textbf{Name}}  & \multicolumn{3}{c}{\textbf{Krylov Its.}}     & \multicolumn{3}{c}{\textbf{max nnz}}      & \multicolumn{3}{c}{\textbf{time (s)}}  \\   \cmidrule(l{2pt}r{2pt}){2-4}\cmidrule(l{2pt}r{2pt}){5-7} \cmidrule(l{2pt}r{2pt}){8-10}
 & {$P^{C}_{AS,\left(k_c,k_r\right)}$} & {$P^{L}_{AS,\left(\lvert \mathcal{N}\rvert,0\right)}$} & {$\widehat{P}_K$} & {$P^{C}_{AS,\left(k_c,k_r\right)}$} & {$P^{L}_{AS,\left(\lvert \mathcal{N}\rvert,0\right)}$} & {$\widehat{P}_K$} & {$P^{C}_{AS,\left(k_c,k_r\right)}$} & {$P^{L}_{AS,\left(\lvert \mathcal{N}\rvert,0\right)}$} & {$\widehat{P}_K$}\\ \midrule
 CVXQP2\_L & 3,241 & 2,057 & 3,078  & $5.06\cdot10^4$ & $2.35\cdot10^6$ & $1.47\cdot10^6$ & 9.01 & 29.74 & 33.52\\
 CVXQP2\_M & 3,019 & 2,588 & 3,128 & $5.04\cdot10^3$ & $1.15\cdot10^5$ & $6.15\cdot10^4$ & 1.02 & 1.54 & 1.44\\
 DUAL3 & 911 & 543 & 992  & $1.12 \cdot 10^2$ &$6.77\cdot 10^3$ & $6.77\cdot 10^3$ & 0.15  & 0.11  & 0.14 \\
   GOULDQP3 & 1,236 & 1,039 & 814 & $4.89 \cdot 10^3$ &$1.05\cdot 10^4$ & $1.05\cdot 10^4$ & 0.31 & 0.37  & 0.27\\
  MOSARQP2 & 752 & 730 & 803  & $1.71 \cdot 10^4$   & $2.18 \cdot 10^4$ &$2.25 \cdot 10^4$  & 0.12  & 0.16  & 0.18\\
 POWELL20 & 1,531 & 1,537 & 1,538  & $2. 10 \cdot 10^4$  & $6.20 \cdot 10^4$ & $6.20 \cdot 10^4$ & 1.64 & 2.41 & 5.24\\
 Q25FV47 & 6,213 & 5,994 & 8,117& $2.13 \cdot 10^4$ & $6.68\cdot 10^4$& $1.14\cdot 10^5$ & 2.12 &3.75 & 4.61\\
QETAMACRO & 4,901  & 4,661 & 6,378  & $1.13\cdot 10^4$ & $2.62\cdot 10^4$ & $2.46 \cdot 10^4$ & 0.71 & 1.40 & 1.60 \\
QISRAEL & 4,516 & 3,367 & 5,920 & $2.17\cdot 10^2$ & $4.60\cdot 10^3$  & $4.84\cdot 10^3$ & 0.49  & 0.61  & 0.80\\
  QSHIP12L & 5,664  & 5,070 & 5,803 &$1.09 \cdot 10^4$ &$2.06\cdot 10^5$ & $2.10\cdot 10^5$ & 3.01 &4.92 & 5.69\\
 STCQP1 &3,410  & 2,543 & 2,691  &$7.00 \cdot 10^5$ &$1.25\cdot 10^5$ & $1.26\cdot 10^5$ & 7.48 & 5.49 & 5.21\\
STCQP2 &3,074  &2,704 & 1,918 &$5.44 \cdot 10^4$ &$2.13\cdot 10^5$ & $2.13\cdot 10^5$ &4.30 & 6.40 & 3.96\\  
 UBH1 & 681 & 679 & 700 & $8.40 \cdot 10^4$ & $2.13\cdot 10^5$ & $2.13\cdot 10^5$ & 1.97 &  3.78& 4.17\\
\bottomrule
\end{tabular}
\end{threeparttable}}
\end{table}
\par From Table \ref{Comparison of QP preconditioners (Maros-Meszaros collection)}, one can observe that most of the time $P^{C}_{AS,\left(k_c,k_r\right)}$ is rather inexpensive, and naturally requires some additional Krylov iterations. On the other hand, $P^{L}_{AS,\left(\lvert \mathcal{N}\rvert,0\right)}$ delivers faster convergence of the Krylov solver, at the cost of additional memory (since we utilize non-diagonal Hessian information). However, while the same is true for most problems when employing $\widehat{P}_K$, the latter can be prone to numerical inaccuracy (since we do not allow the use of $2\times 2$ pivots in the $LDL^\top$ factorization). Whether the use of non-diagonal Hessian information is beneficial should depend on the problem under consideration. In the above experiments, this proved to be beneficial for only 4 out of the 13 instances tested (that is DUAL3, GOULDQP3, STCQP1, STCQP2). Nevertheless, we can observe that all three approaches are competitive, while $P^{C}_{AS,\left(k_c,k_r\right)}$ and $P^{L}_{AS,\left(\lvert \mathcal{N}\rvert,0\right)}$ are both very stable. 
\paragraph{PDE-constrained optimization instances.}
\par Next, we compare the preconditioning approaches on some Partial Differential Equation (PDE) optimization problems. In particular, we consider the $\text{L}^1/\text{L}^2$-regularized Poisson control problem, as well as the $\text{L}^1/\text{L}^2$-regularized convection--diffusion control problem with control bounds. We should emphasize at this point that while beskope preconditioners have been created for PDE problems of this form, here we treat the discretized problems as if we hardly knew anything about their structure, to demonstrate the generality of the approaches presented in this paper.
\par We consider problems of the following form:
\begin{equation} \label{generic inverse problem}
\begin{split}
\min_{\mathrm{y},\mathrm{u}} \       &\ \mathrm{J}(\mathrm{y}(\bm{x}),\mathrm{u}(\bm{x})) \coloneqq \frac{1}{2}\| \rm{y} - \bar{\rm{y}}\|_{L^2(\Omega)}^2 + \frac{\alpha_1}{2}\|\rm{u}\|_{L^1(\Omega)} + \frac{\alpha_2}{2}\|\rm{u}\|_{L^2(\Omega)}^2, \\
\text{s.t.}\ &\ \mathrm{D} \mathrm{y}(\bm{x})  = \mathrm{u}(\bm{x}) + \mathrm{g}(\bm{x}),\\
       &\ \mathrm{u_{a}}(\bm{x}) \leq \mathrm{u}(\bm{x})  \leq \mathrm{u_{b}}(\bm{x}),
\end{split}
\end{equation}
\noindent where $(\rm{y},\rm{u}) \in \text{H}^1(\Omega) \times \text{L}^2(\Omega)$, $\mathrm{D}$ denotes some linear differential operator associated with the differential equation, $\bm{x}$ is a $2$-dimensional spatial variable, and $\alpha_1,\ \alpha_2 \geq 0$ denote the regularization parameters of the control variable. We note that other variants for $\mathrm{J}({\rm y},{\rm u})$ are possible, including measuring the state misfit and/or the control variable in other norms, as well as alternative weightings within the cost functionals. In particular, the methods tested here also work well for $\text{L}^2$-norm problems (e.g. see \cite{PearsonGondzioNumMath}). We consider problems of the form of \eqref{generic inverse problem} to create an extra level of difficulty for our solvers.
\par The problem is considered on a given compact spatial domain $\Omega$, where $\Omega \subset \mathbb{R}^{2}$ has boundary $\partial \Omega$, and is equipped with Dirichlet boundary conditions. The algebraic inequality constraints are assumed to hold a.e. on $\Omega$. We further note that ${\rm u_a}$ and ${\rm u_b}$ may take the form of constants, or functions in spatial variables, however we restrict our attention to the case where these represent constants. 
\par Problems in the form of \eqref{generic inverse problem} are often solved numerically, by means of a discretization method. In the following experiments we employ the Q1 finite element discretization implemented in IFISS\footnote{\url{https://personalpages.manchester.ac.uk/staff/david.silvester/ifiss/default.htm}} (see \cite{IFISSACM,IFISSSIAMREVIEW}). Applying the latter yields a sequence of non-smooth convex programming problems, which can be transformed to the smooth form of \eqref{QP problem}, by introducing some auxiliary variables to deal with the $\ell_1$ terms appearing in the objective (see \cite[Section 2]{PearsonPorcelliStollNLAA}). In order to restrict the memory requirements of the approach, we consider an additional approximation of $H$ in the preconditioner $P^{L}_{AS,\left(\lvert \mathcal{N}\rvert,0\right)}$. In the cases under consideration, the resulting Hessian matrix takes the following form:
\[H = \begin{bmatrix}
J_M &0_{d,d} & 0_{d,d}\\
0_{d,d} & \alpha_2 J_M & -\alpha_2 J_M\\
0_{d,d} & -\alpha_2 J_M & \alpha_2 J_M
\end{bmatrix}, \]
\noindent where $J_M$ is the mass matrix of size $d$. When non-diagonal Hessian information is utilized within the preconditioner, we approximate each block of $H$ by its diagonal (i.e. $\widetilde{J}_M = \textnormal{Diag}(J_M)$; an approximation which is known to be optimal \cite{WathenIMANumAnal}). The resulting matrix is then further approximated as discussed in Section \ref{Section: regularized saddle point matrices}. From now on, the $LDL^\top$ preconditioner, which is based on an approximation of $P^{L}_{AS,\left(\lvert \mathcal{N}\rvert,0\right)}$, is referred to as $\widehat{P}^{L}_{AS,\left(\lvert \mathcal{N}\rvert,0\right)}$, in order to stress the additional level of approximation employed within the Hessian matrix. For these examples, the preconditioning strategy based on $\widehat{P}_K$ (given in Section \ref{Subsection: factorization-based preconditioners for saddle point systems}) behaved significantly worse, and hence was not included in the numerical results. The preconditioner $\widehat{P}^{L}_{AS,\left(\lvert \mathcal{N}\rvert,0\right)}$ can be useful in that it allows us to employ block-diagonal preconditioners of which the Schur complement takes into account non-diagonal information of the Hessian matrix $H$. In certain cases, this can result in a faster convergence of IP-PMM, as compared to ${P}^{C}_{AS,\left(k_c,k_r\right)}$ (see Table \ref{Table Convection-Diffusion Control Problem: increasing grid-size}).
\par The first problem that we consider is the two-dimensional $\text{L}^1/\text{L}^2$-regularized Poisson optimal control problem, with bound constraints on the control and free state, posed on the domain $\Omega = (0,1)^2$. Following \cite[Section 5.1]{PearsonPorcelliStollNLAA}, we consider the constant control bounds $\rm{u_a} = -2$, $\rm{u_b} = 1.5$, and the desired state $\bar{\rm{y}} = \sin(\pi x_1)\sin(\pi x_2)$. In Table \ref{Table Poisson Control Problem: increasing grid-size}, we fix $\alpha_2 = 10^{-2}$ (which we find to be the most numerically interesting case), and we present the runs of the method using the different preconditioning approaches, with increasing problem size, and varying $\text{L}^1$ regularization parameter (that is $\alpha_1$). To reflect the change in the grid size, we report the number of variables of the optimization problem after transforming it to the IP-PMM format. We also report the overall number of Krylov iterations required for IP-PMM to converge (and the number of IP-PMM iterations in brackets), the maximum number of non-zeros stored in order to apply the inverses of the associated preconditioners, as well as the required CPU time.
\begin{table}[!ht]
\centering
\caption{Comparison of QP preconditioners (Poisson Control: problem size and varying regularization).\label{Table Poisson Control Problem: increasing grid-size}}
\scalebox{0.9}{
\begin{tabular}{llrrrrrr}     
\toprule
    \multirow{2}{*}{\textbf{$\bm{n}$}} & \multirow{2}{*}{$\bm{\alpha_1}$} & \multicolumn{2}{c}{\textbf{Krylov (IP) Its.}}     & \multicolumn{2}{c}{\textbf{max nnz}}      & \multicolumn{2}{c}{\textbf{time (s)}}  \\   \cmidrule(l{2pt}r{2pt}){3-4}\cmidrule(l{2pt}r{2pt}){5-6} \cmidrule(l{2pt}r{2pt}){7-8}
&   &  {${P}^{C}_{AS,\left(k_c,k_r\right)}$} & {$\widehat{P}^{L}_{AS,\left(\lvert \mathcal{N}\rvert,0\right)}$}  & {${P}^{C}_{AS,\left(k_c,k_r\right)}$} & {$\widehat{P}^{L}_{AS,\left(\lvert \mathcal{N}\rvert,0\right)}$} & {${P}^{C}_{AS,\left(k_c,k_r\right)}$} & {$\widehat{P}^{L}_{AS,\left(\lvert \mathcal{N}\rvert,0\right)}$} \\ \midrule
 \multirow{3}{*}{$2.11\cdot 10^4$} & $10^{-2}$ & 1,353 (13) & 868 (13) & $3.88\cdot10^5$ & $4.65\cdot10^5$  & 5.81  & 5.91  \\
  &   $10^{-4}$ & 1,586 (14)& 1,015 (14)& $3.88\cdot10^5$ & $4.65\cdot10^5$  & 6.53 & 6.58 \\
  &   $10^{-6}$ & 1,586 (14)& 1,013 (14)& $3.88\cdot10^5$ & $4.65\cdot10^5$  & 6.75 & 6.61 \\ \hdashline
 \multirow{3}{*}{$8.32\cdot 10^4$} & $10^{-2}$ & 1,327 (14) & 887 (14)  & $2.12\cdot10^6$ & $2.21\cdot10^6$ & 22.67 & 39.60 \\
 &   $10^{-4}$ & 1,759 (15) & 1,054 (15) & $2.12\cdot10^6$ & $2.21\cdot10^6$ & 27.58 & 44.14 \\
  &   $10^{-6}$ & 1,575 (14)& 934 (14) & $2.12\cdot10^6$ & $2.21\cdot10^6$ & 24.83 & 40.23 \\ \hdashline
  \multirow{3}{*}{$3.30\cdot 10^5$}  & $10^{-2}$ & 246 (8) & 203 (8) & $1.06\cdot10^7$ & $1.09\cdot10^7$  & 27.52 & 75.89 \\
   &   $10^{-4}$& 246 (8) & 204 (8)& $1.06\cdot10^7$ & $1.09\cdot10^7$  & 27.20 & 76.73 \\
   &   $10^{-6}$ & 246 (8) & 204 (8)& $1.06\cdot10^7$ & $1.09\cdot10^7$  & 27.38 & 77.15 \\ \hdashline
   \multirow{3}{*}{$1.32\cdot 10^6$} & $10^{-2}$ & 193 (7) & 158 (7) & $5.51\cdot10^7$ & $5.38\cdot10^7$  & 99.14 & 308.79\\ 
       &   $10^{-4}$ & 193 (7) & 158 (7) & $5.51\cdot10^7$ & $5.38\cdot10^7$  & 101.78 & 318.35\\
   &   $10^{-6}$ & 193 (7) & 158 (7) & $5.51\cdot10^7$ & $5.38\cdot10^7$ & 99.71 & 318.01\\
\bottomrule
\end{tabular}}
\end{table}

\par We can draw several observations from the results in Table \ref{Table Poisson Control Problem: increasing grid-size}. Firstly, one can observe that in this case, a diagonal approximation of $H$ is sufficiently good to deliver very fast convergence of MINRES. The block-diagonal preconditioner using non-diagonal Hessian information (i.e. $\widehat{P}^{L}_{AS,\left(\lvert \mathcal{N}\rvert,0\right)}$) required consistently fewer MINRES iterations (and not necessarily more memory; see the three largest experiments), however, this did not result in a reduction in CPU time. There are several reasons for this. Firstly, the Hessian of the problem becomes ``no less" diagonally dominant as the problem size is increased. As a result, the diagonal approximation of it remains robust with respect to the problem size for the problem under consideration. On the other hand, the algorithm uses the built-in MATLAB function \texttt{ldl} to factorize the preconditioner $\widehat{P}^{L}_{AS,\left(\lvert \mathcal{N}\rvert,0\right)}$. While this implementation is very stable, it employs a dynamic permutation at each IP-PMM iteration, which slows down the algorithm. In this case, a specialized method using preconditioner $\widehat{P}^{L}_{AS,\left(\lvert \mathcal{N}\rvert,0\right)}$ should employ a separate symbolic factorization step, that could be used in subsequent IP-PMM iterations, thus significantly reducing the CPU time. This is not done here, however, as we treat these PDE-optimization problems as black-box (notice that the implementation allows the user to feed an approximation of the Hessian, but does not allow the user to use a different $LDL^\top$ decomposition). In all the previous runs, the reported Krylov iterations include both the predictor and the corrector steps of IP-PMM. Thus, the systems solved in each case are twice the number of IP iterations. We should note that for the problem under consideration employing a predictor--corrector scheme is not necessary, however, we wanted to keep the implementation as general and robust as possible, without tailoring it to specific applications. For this problem, we can also observe that IP-PMM was robust with respect to the problem size (i.e. IP-PMM convergence was not significantly affected by the size of the problem). This is often observed when employing an IPM for the solution of PDE optimization problems (e.g. see \cite{PearsonGondzioNumMath}), however, in theory one should expect dependence of IPM on the problem size.
\par Next we consider the optimal control of the convection--diffusion equation, i.e. $-\epsilon \rm{\Delta y} + \rm{w} \cdot \nabla y = u$, on the domain $\Omega = (0,1)^2$, where $\rm{w}$ is the wind vector given by $\rm{w} = [2x_2(1-x_1)^2, -2x_1(1-x_2^2)]^\top$, with control bounds $\rm{u_a} = -2$, $\rm{u_b} = 1.5$ and free state (e.g. see \cite[Section 5.2]{PearsonPorcelliStollNLAA}). Once again, the problem is discretized using Q1 finite elements, employing the Streamline Upwind Petrov-Galerkin (SUPG) upwinding scheme implemented in \cite{BrooksHughesCMAME}. We define the desired state as $\rm{\bar{y}} = exp(-64((x_1 - 0.5)^2 + (x_2 - 0.5)^2))$ with zero boundary conditions. The diffusion coefficient $\epsilon$ is set as $\epsilon = 0.02$. The $\text{L}^2$ regularization parameter $\alpha_2$ is set as $\alpha_2 = 10^{-2}$. We run IP-PMM with the two different preconditioning approaches on the aforementioned problem, with different $\text{L}^1$ regularization values (i.e. $\alpha_1$) and with increasing problem size. The results are collected in Table \ref{Table Convection-Diffusion Control Problem: increasing grid-size}.
\begin{table}[!ht]
\centering
\caption{Comparison of QP preconditioners (Convection--Diffusion Control: problem size and varying regularization).\label{Table Convection-Diffusion Control Problem: increasing grid-size}}
\scalebox{0.9}{
\begin{tabular}{llrrrrrr}     
\toprule
    \multirow{2}{*}{\textbf{$\bm{n}$}} & \multirow{2}{*}{$\bm{\alpha_1}$} & \multicolumn{2}{c}{\textbf{Krylov (IP) Its.}}     & \multicolumn{2}{c}{\textbf{max nnz}}      & \multicolumn{2}{c}{\textbf{time (s)}}  \\   \cmidrule(l{2pt}r{2pt}){3-4}\cmidrule(l{2pt}r{2pt}){5-6} \cmidrule(l{2pt}r{2pt}){7-8}
&   &  {${P}^{C}_{AS,\left(k_c,k_r\right)}$} & {$\widehat{P}^{L}_{AS,\left(\lvert \mathcal{N}\rvert,0\right)}$}  & {${P}^{C}_{AS,\left(k_c,k_r\right)}$} & {$\widehat{P}^{L}_{AS,\left(\lvert \mathcal{N}\rvert,0\right)}$} & {${P}^{C}_{AS,\left(k_c,k_r\right)}$} & {$\widehat{P}^{L}_{AS,\left(\lvert \mathcal{N}\rvert,0\right)}$} \\ \midrule
 \multirow{3}{*}{$2.11\cdot 10^4$} & $10^{-2}$ & 3,947 (21) & 1,903 (19) & $3.88\cdot10^5$ & $4.65\cdot10^5$  & 15.94 & 11.67 \\
  &   $10^{-4}$ & 7,546 (25) & 2,721 (22)& $3.88\cdot10^5$ & $4.65\cdot10^5$  &  31.24 & 15.31 \\
  &   $10^{-6}$ & 7,489 (25) & 2,962 (23) & $3.88\cdot10^5$ & $4.65\cdot10^5$  & 31.29 & 16.42 \\ \hdashline
 \multirow{3}{*}{$8.32\cdot 10^4$} & $10^{-2}$ & 3,464 (19) & 1,937 (19) & $2.12\cdot10^6$ & $2.21\cdot10^6$ & 49.05 & 65.98 \\
 &   $10^{-4}$ & 7,198 (25) & 2,976 (23)  & $2.12\cdot10^6$ & $2.21\cdot10^6$ & 99.49 & 93.19 \\
  &   $10^{-6}$ & 7,150 (25) & 3,178 (24) & $2.12\cdot10^6$ & $2.21\cdot10^6$ & 98.82 &  98.90 \\ \hdashline
  \multirow{3}{*}{$3.30\cdot 10^5$}  & $10^{-2}$ & 4,037 (21) & 1,667 (18) & $1.07\cdot10^7$ & $1.09\cdot10^7$  & 297.96 & 285.77 \\
   &   $10^{-4}$& 4,971 (22) & 2,542 (22) & $1.07\cdot10^7$ & $1.09\cdot10^7$  & 357.43 & 334.89 \\
   &   $10^{-6}$ & 5,418 (23) & 2,530 (22) & $1.07\cdot10^7$ & $1.09\cdot10^7$  & 372.40 & 354.23\\ \hdashline
   \multirow{3}{*}{$1.32\cdot 10^6$} & $10^{-2}$ & 384 (9) & 284 (9) & $5.51\cdot10^7$ & $5.38\cdot10^7$  & 158.17 & 436.49\\ 
       &   $10^{-4}$ & 385 (9) & 286 (9) & $5.51\cdot10^7$ & $5.38\cdot10^7$  & 161.12 & 438.44\\
   &   $10^{-6}$ & 385 (9) & 286 (9) & $5.51\cdot10^7$ & $5.38\cdot10^7$  & 165.57 & 443.67\\
\bottomrule
\end{tabular}}
\end{table}
\par In Table \ref{Table Convection-Diffusion Control Problem: increasing grid-size} we can observe that the IP-PMM convergence is improved when the problem size is increased, which relates to the good conditioning of the Hessian. On the other hand, IP-PMM convergence is affected by the $\text{L}^1$ regularization parameter $\alpha_1$. Unlike in the Poisson control problem, we can see clear advantages of using $\widehat{P}^{L}_{AS,\left(\lvert \mathcal{N}\rvert,0\right)}$ instead of ${P}^{C}_{AS,\left(k_c,k_r\right)}$ in this case. We can observe that in this problem using non-diagonal Hessian information within the preconditioner is significantly more important, and the reduced number of Krylov iterations often translates into a reduction of the CPU time. As before, we should mention that the reported number of Krylov iterations includes the solution of both the predictor and the corrector steps for each IP-PMM iteration.
\par Overall, we observe that each of the presented approaches can be very successful on a wide range of problems, including those of very large scale. Although we have treated every problem as if we knew nothing about its structure for these numerical tests, our a priori knowledge of the preconditioners and of the problem's structure could in principle aid us in selecting a preconditioner without compromising their ``general purpose" nature.

\section{Conclusions} \label{Section Conclusions}
In this paper we have presented several general-purpose preconditioning methodologies suitable for primal-dual regularized interior point methods, applied to convex optimization problems. All presented preconditioners are positive definite and hence can be used within symmetric solvers such as PCG or MINRES. After analyzing and discussing the different preconditioning approaches, we have presented extensive numerical results, showcasing their use and potential benefits for different types of practical applications of convex optimization. A robust and general IP-PMM implementation, using the proposed preconditioners, has been provided for the solution of convex quadratic programming problems, and one can readily observe its ability of reliably and efficiently solving general large-scale problems, with minimal input from the user. 
\par As a future research direction, we would like to include certain matrix-free preconditioning methodologies that could be used as alternatives for huge-scale instances that cannot be solved by means of factorization-based preconditioners, due to memory requirements.

\section*{Acknowledgments} 
JG and SP were supported by the Google project ``Fast (1 + $x$)-order Methods for Linear Programming". JWP gratefully acknowledges support from the Engineering and Physical Sciences Research Council (UK) grant EP/S027785/1. We wish to remark that this study does not have any conflict of interest to disclose.

\appendix
\section{Proof of Theorem \ref{Theorem Spectral analysis augmented system}} \label{Appendix: proof of spectral analysis for augmented system}
For simplicity of notation, let $P_{AS,(k_c,k_r)} \equiv P_{AS}$. In order to provide an outline of the proof of Theorem \ref{Theorem Spectral analysis augmented system}, which follows trivially by extending the result in \cite[Theorem 3]{NLAA:BergGondMartPearPoug}, we need to introduce the notion of the Rayleigh quotient for symmetric matrices. The numerical range of a symmetric matrix $U$, denoted as $q(U)$, is defined as 
\[q(U) \coloneqq \left\{z \in \mathbb{R}, \textnormal{ s.t.}\ z = \frac{x^\top U x}{x^\top x},\textnormal{ for some } x \in \mathbb{R}^n,\ x\neq 0\right\}.\]
\par Given the notation of Section \ref{Subsection: block diagonal preconditioners}, an  element of the Rayleigh quotient of these matrices is represented as:
\[\begin{array}{lcllcl}
	\gamma_{NE} & \in & q\left(\widetilde{M}_{NE}\right) = [\alpha_{NE}, \beta_{NE}],  ~~
	& ~~\gamma_{F}  &\in &q\left(\widetilde{F}\right) =  [\alpha_{F}, \beta_{F}]. 
\end{array}\] 
\noindent Similarly, an arbitrary element of $q(P_{NE})$ is denoted by
\[\begin{array}{lcllclcl}
\gamma_{p} & \in & [\lambda_{\min}(P_{NE}),\lambda_{\max} (P_{NE})] &\subseteq & \left[\delta,\dfrac{\sigma_{\max}^2(A)}{\rho} + \delta\right].
	\end{array}\]
		  The eigenvalues of $P_{AS}^{-1} K$ are the same as those of  
		\[P_{AS}^{-1/2}K P_{AS}^{-1/2}  = 
		\begin{bmatrix} \widehat{F}^{-1/2} & 0_{n,m} \\ 0_{m,n} & P_{NE}^{-1/2}  \end{bmatrix}
\begin{bmatrix} -F & A^\top \\ A & \delta I_m \end{bmatrix}
			\begin{bmatrix} \widehat{F}^{-1/2} & 0_{n,m} \\ 0_{m,n} & P_{NE}^{-1/2}  \end{bmatrix}
				=
	\begin{bmatrix} -\widetilde{F} & R^\top \\ R & 
	\delta P_{NE}^{-1} 	\end{bmatrix},  \]
			where  $\widetilde{F}$ is defined in \eqref{simsym} and
			$R \coloneqq P_{NE}^{-1/2} A \widehat{F}^{-1/2}.$
Any eigenvalue $\lambda$ of the preconditioned matrix $P_{AS}^{-1/2} K P_{AS}^{-1/2} $
		must therefore satisfy
		\begin{eqnarray}
			\label{V-first}
		-\widetilde{F} w_1\ +&  R^\top w_2 &= \lambda w_1, \\
			\label{V-second}
		R w_1\ +&   \delta P_{NE}^{-1} w_2  &= \lambda w_2. \end{eqnarray}
		First, note that
		\begin{equation*}
			 \scalemath{0.9}{ R R^\top = P_{NE}^{-1/2} A \widehat{F}^{-1} A^\top P_{NE}^{-1/2} = P_{NE}^{-1/2} \left(\widehat{M} - \delta I_m\right) P_{NE}^{-1/2} = \widetilde{M}_{NE} -
		\delta P_{NE}^{-1}. }
		\end{equation*}
		If $\lambda  \not \in [-\beta_{F}, -\alpha_{F}]$ then $\widetilde{F} + \lambda I_n$ is symmetric positive (or negative) definite;
		moreover $R^\top w_2 \ne 0_n$.
		Then from \eqref{V-first} we obtain
		\[ w_1 = \left(\widetilde{F} + \lambda I_n\right)^{-1} R^\top w_2, \]
		which, after substituting in \eqref{V-second}, yields
		\[ R  \left(\widetilde{F} + \lambda I_n\right)^{-1} R^\top  w_2  + \delta P_{NE}^{-1} w_2= \lambda w_2. \]
		Premultiplying by $w_2^\top$ and dividing by $\|w_2\|^2$, we obtain the following equation, where we have set $z = R^\top w_2$:
		\[ \lambda  = \frac{z^\top \left(\widetilde{F} + \lambda I_n\right)^{-1} z}{z^\top z}
		\frac {w_2^\top R R^\top w_2}{w_2^\top w_2}  +\delta \frac{w_2^\top P_{NE}^{-1} w_2}{w_2^\top w_2} = \frac{1}{\gamma_{F} + \lambda} \left(\gamma_{NE} - \frac{\delta}{\gamma_p}\right)
		+  \frac{\delta}{\gamma_{p}}. \]
		Hence, $\lambda$ must satisfy the following second-order algebraic equation:
		\[ \lambda^2 +  \left(\gamma_{F} -\omega \right)\lambda-\left(\omega (\gamma_F - 1)+\gamma_{NE}\right) = 0,\]
		where we have set $\omega = \dfrac{\delta}{\gamma_{p}}$ satisfying $\omega \leq 1$.  Notice that $\gamma_{NE} - \omega \geq 0$ by construction.
\par All bounds, except for the lower bound of $I_{+}$, follow directly by following the developments in \cite[Theorem 3]{NLAA:BergGondMartPearPoug}. Thus, we only derive the lower bound for the positive eigenvalues of the preconditioned matrix, which can be obtained by computing a lower bound for the positive eigenvalue solution of the previous algebraic equation. In particular, we have
\begin{equation*}
\begin{split}
\lambda_+ =& \frac{1}{2}\left[\omega-\gamma_{F} + \sqrt{(\gamma_{F}-\omega)^2+ 4(\omega \gamma_{F} -\omega + \gamma_{NE})}\right] \\=& \frac{1}{2}\left[\omega-\gamma_{F} + \sqrt{(\gamma_{F}+\omega)^2+ 4(\gamma_{NE}
-\omega) }\right].
\end{split}
\end{equation*}
We notice that $\lambda_+$ is a decreasing function with respect to the variable $\gamma_{F}$ and increasing with respect to $\gamma_{NE}$. Hence, we have that:
\begin{equation*}
\begin{split}
	\lambda_+ \geq \ & \frac{1}{2}\left[\omega - \beta_{F} + \sqrt{(\beta_{F} + \omega)^2 + 4(\alpha_{NE}-\omega)} \right]	\\
 \geq \ & \frac{1}{2}\left[ -\beta_{F} + \sqrt{\beta_{F}^2 + 4\alpha_{NE}}\right], 
\end{split}
\end{equation*}
where the last inequality follows because the penultimate expression is increasing with respect to $\omega$.

\bibliography{references} 
\bibliographystyle{siamplain}

\end{document}